\documentclass[3p,11pt]{elsarticle}

\usepackage{amsmath}
\usepackage{tikz}
\usepackage{relsize}
\usetikzlibrary{calc}
\usepackage{url}

\usepackage{amssymb,latexsym, amsthm, pifont}
\usepackage{cancel}
\usepackage{bbold}

\usepackage{verbatim}

\usepackage{graphicx}
\usepackage{color}

\usepackage{nomencl} 
\usepackage{hyperref}

\theoremstyle{plain}
\newtheorem{lemma}{Lemma}

\newtheorem{theorem}[lemma]{Theorem}
\newtheorem{corollary}[lemma]{Corollary}

\theoremstyle{definition}
\newtheorem{definition}[lemma]{Definition}

\theoremstyle{remark}
\newtheorem{remark}[lemma]{Remark}
\newtheorem{claim}[lemma]{Claim}

\newtheorem{notation}[lemma]{Notation}

\numberwithin{equation}{section}

\makenomenclature

\makeatletter
\let\pgfmath@function@exp\relax 
\pgfmathdeclarefunction{exp}{1}{%
  \begingroup
    \pgfmath@xc=#1pt\relax
    \pgfmath@yc=#1pt\relax
    \ifdim\pgfmath@xc<-9pt
      \pgfmath@x=1sp\relax
    \else
      \ifdim\pgfmath@xc<0pt
        \pgfmath@xc=-\pgfmath@xc
      \fi
      \pgfmath@x=1pt\relax
      \pgfmath@xa=1pt\relax
      \pgfmath@xb=\pgfmath@x
      \pgfmathloop%
        \divide\pgfmath@xa by\pgfmathcounter
        \pgfmath@xa=\pgfmath@tonumber\pgfmath@xc\pgfmath@xa%
        \advance\pgfmath@x by\pgfmath@xa
      \ifdim\pgfmath@x=\pgfmath@xb
      \else
        \pgfmath@xb=\pgfmath@x
      \repeatpgfmathloop%
      \ifdim\pgfmath@yc<0pt
        \pgfmathreciprocal@{\pgfmath@tonumber\pgfmath@x}%
        \pgfmath@x=\pgfmathresult pt\relax
      \fi
    \fi
    \pgfmath@returnone\pgfmath@x%
  \endgroup
}
\makeatother

\colorlet{gialloLimone}{yellow!80}

\newcommand{\R}{\mathbb{R}}
\newcommand{\N}{\mathbb{N}}

\newcommand{\Id}{\mathbb{I}}
\newcommand{\Z}{\mathbb{Z}}

\newcommand{\D}{\mathcal{D}}

\newcommand{\distrBdd}{ L^{\infty}}
\newcommand{\pw}{\mathfrak L^{\infty}}
\newcommand{\pwu}{\mathfrak L_{B}^{\infty}}
\newcommand{\pwx}{\mathfrak L_{L}^{\infty}}

\newcommand{\Gfr}{\mathfrak G}
\newcommand{\gfr}{{{\mathfrak g}}}

\newcommand{\gx}{{\mathpzc{g}_{L}}}
\newcommand{\gu}{{\mathcalligra{g}_{B}}}
\newcommand{\gsimple}{{g}}
\newcommand{\gd}{{g_{E}}}

\newcommand{\loc}{\text{\rm loc}}

\newcommand{\Ll}{\mathcal L}

\newcommand{\Ha}{\mathcal H}

\newcommand{\rc}{\mathrm c}
\newcommand{\norm}[1]{\lVert#1\rVert}

\newcommand{\paytau}[1]{\frac{\partial#1}{\partial{\ytau}}}

\newcommand{\pax}[1]{\frac{\partial#1}{\partial{x}}}
\newcommand{\pat}[1]{\frac{\partial#1}{\partial{t}}}
\newcommand{\reu}{\mathrm e_{1}}
\newcommand{\red}{\mathrm e_{2}}

\DeclareMathOperator{\diam}{diam}
\DeclareMathOperator{\infl}{Inf{}l}
\DeclareMathOperator{\ri}{r.i.}
\DeclareMathOperator{\clos}{clos}

\DeclareMathOperator{\BV}{BV}

\DeclareMathOperator{\Dt}{D_{\mathit t}}
\DeclareMathOperator{\Dx}{D_{\mathit x}}
\DeclareMathOperator{\Dif}{D}
\newcommand{\ytau}{y}
\DeclareMathOperator{\Dytau}{D_{\mathit y}}

\DeclareMathOperator{\pytau}{\partial_{\mathit\ytau}}

\DeclareMathOperator{\pt}{\partial_{\mathit t}}
\DeclareMathOperator{\px}{\partial_{\mathit x}}
\DeclareMathOperator{\ddt}{\frac{d}{d\mathit t}}
\DeclareMathOperator{\dds}{\frac{d}{d\mathit s}}

\DeclareMathAlphabet{\mathcalligra}{T1}{calligra}{m}{n}
\DeclareMathAlphabet{\mathpzc}{OT1}{pzc}{m}{it}

\newcounter{stepnb}
\newcounter{substepnb}
\newcommand{\firststep}{\setcounter{stepnb}{0}}
\newcommand{\firstsubstep}{\setcounter{substepnb}{0}}
\newcommand{\step}[1]{\vskip.3\baselineskip\emph{\addtocounter{stepnb}{1} \arabic{stepnb}: #1.}  }
\newcommand{\substep}[1]{\vskip.3\baselineskip\emph{\addtocounter{substepnb}{1} \arabic{stepnb}.\arabic{substepnb}: #1.}  }

\begin{document}

\begin{frontmatter}

\title{Eulerian, Lagrangian and Broad continuous solutions\\ to a balance law with non convex f{}lux I}

\author[GA]{G.~Alberti}
\ead{galberti1@dm.unipi.it}
\address[GA]{
Dipartimento di Matematica,
Universit\`a di Pisa,
largo Pontecorvo 5,
56127 Pisa,
Italy}

\author[SB]{S.~Bianchini}
\ead{bianchin@sissa.it}
\address[SB]{
SISSA-ISAS, Via Bonomea 265, 34136 Trieste, Italy}

\author[LC]{L.~Caravenna}
\ead{laura.caravenna@unipd.it}
\address[LC]{
Dipartimento di Matematica `Tullio Levi Civita', Universit\`a di Padova, via Trieste 63, 35121 Padova, Italy}

\begin{abstract}
We discuss different notions of \emph{continuous} solutions to the balance law \[\pt u  + \px (f(u )) =\gsimple \qquad\text{$\gsimple$ bounded, $f\in C^{2}$}\] extending previous works relative to the f{}lux $f(u)=u^{2}$. 
We establish the equivalence among distributional solutions and a suitable notion of Lagrangian solutions for general smooth f{}luxes.
{We eventually find that continuous solutions are Kruzkov iso-entropy solutions, which yields uniqueness for the Cauchy problem.} We also reduce the ODE on \emph{any} characteristics under the sharp assumption that the set of inf{}lection points of the f{}lux $f$ is negligible.
The correspondence of the source terms in the two settings is matter of the companion work~\cite{file2ABC}, where we include counterexamples when the negligibility on inf{}lection points fails.
\end{abstract}

\begin{keyword}
Balance law, Lagrangian description, Eulerian formulation. 
\MSC 35L60, 37C10, 58J45
\end{keyword}

\end{frontmatter}

\tableofcontents

\section{Introduction}

Single balance laws in one space dimension mostly present smooth f{}luxes, although the case of piecewise smooth f{}luxes is of interest both for the mathematics and for applications.
Source terms instead are naturally rough, and singularities of different nature have a physical and geometrical meaning.
As well, they might indeed make a difference among the Eulerian and Lagrangian description of the phenomenon which is being modeled, for the mathematics.

We are concerned in this paper with different notions of \emph{continuous} solutions of the PDE
\begin{equation}
\label{E:basicPDE}
\pt u(t,x) + \px (f(u(t,x))) =  \gsimple(t,x)
\qquad f\in C^{2}(\R)
\end{equation}
for a bounded source term $ \gsimple$.
{An essential feature of conservation laws is that solutions to the Cauchy problem do develop shocks in finite time. Nevertheless, the source might act as a control device for preventing this shock formation: exploiting the geometric interplay and correspondence with intrinsic Lipschitz surfaces in the Heisenberg group, \cite{FSSC,BCSC} show, for the quadratic flux, that the Cauchy problem admits continuous solutions for any H\"older continuous initial datum, if one chooses accordingly a bounded source term.
This framework of continuous solutions, with more regularity assumptions on the source term, was already considered in~\cite{Daf} as the natural class of solutions to certain interesting dispersive partial differential equations that can be recast as balance laws.
We believe that our study is relevant also in order to point out that, even in the analysis of a single equation in one space dimension, the mathematical difficulties do not only arise by the presence of shocks: also the study of continuous solutions has important delicate points which are not technicalities. This fails the expectation that the study of continuous solutions should be easy, and equation~\eqref{E:basicPDE} is a toy-model for more complex situations.
}

One can adopt the Eulerian viewpoint or the Lagrangian/Broad viewpoint: roughly, the first interprets the equation in a distributional sense while the second consists in an infinite dimensional system of ODEs along characteristics. We compare here the equivalence among the formulations when $u$ is assumed to be continuous, but no more. We remark that even with the quadratic f{}lux \[f(z)=z^{2}/2\] in general $u$ is not more than H\"older continuous, see~\cite{Kirch}, so that a finer analysis is needed.
Continuous solutions are regularized to locally Lipschitz on the open set $\{f'(u)f''(u)\neq0\}$, exploiting the results of this paper, for time-dependent solutions when the source term $\gsimple$ is autonomous~\cite{autonABC}, but not in general.
Examples of stationary solutions which are neither absolutely continuous nor of bounded variation are trivially given by continuous functions $u(x)$ for which $f(u(x))$ has bounded derivative.
Correspondences among different formulations are already done at different levels in~\cite{Daf,Vitt,BSC,BCSC, Pinamonti} for the special, but relevant, case of the quadratic f{}lux. We extend the analysis with new tools.
The issue is delicate because $ \gsimple$ in this setting lacks even of continuity in the $x$ variable, and characteristic curves need not be unique because $u$ lacks of smoothness.
As a consequence, the source terms for the two descriptions lie in different spaces:
\begin{itemize}
\item
In the Eulerian point of view $\gsimple$ is identified only as a distribution in the $(t,x)$-space. 
\item In the Lagrangian/Broad viewpoint it is the restriction of $\gsimple$ on any characteristic curve which must identify uniquely a distribution in the $t$-space---or for a weaker notion only on a chosen family of characteristics that we call Lagrangian parameterization. 
\end{itemize}

The aim of this paper is to consider and to discuss when Eulerian, Broad and Lagrangian solution of~\eqref{E:basicPDE} that we just mentioned are equivalent notions, without addressing what is the correspondence among the suitable source terms---if any.
The correspondence of the source terms, source terms which belong to different functional spaces, is the subject of the companion paper~\cite{file2ABC}, including counterexamples which show that the formulations are not always equivalent.

{We conclude mentioning that Broad solutions were introduced in~\cite{RY} as generalizations of classical solutions alternative to the distributional (Eulerian) ones, and presented e.g.~also in~\cite{Bre}. They were successfully studied and applied in different situations where characteristic curves are unique; the analysis in situations when characteristics do merge and split however was only associated to the presence of shocks, and a different analysis related to multivalued solutions was performed. They were then considered for the quadratic flux by F.~Bigolin and F.~Serra Cassano for their interest related to intrinsic regular and intrinsic Lipschitz surfaces in the Heisenberg group. Our notions of Lagrangian and Broad solutions collapse and substantially coincide with the ones in the literature when the settings overlap. They are otherwise a nontrivial extension of those concepts, and most of the issues in the analysis arise because of our different setting.}

\subsection{Definitions and Setting}
As we are in a non-standard setting, we explain extensively the different notions of solutions and we specify the notation we adopt.
Even if this is an heavy block, detailed definitions improve the later analysis.
They will be also collected in the Nomenclature at the end for an easy consultation.

\nomenclature{$\Ll^{1}$, $\Ll^{2}$}{1- or 2-dimensional Lebesgue measure}
\nomenclature{$f$}{Flux function for the balance law~\eqref{E:basicPDE}}
\nomenclature{$u$}{Continuous solution, Noation~\ref{N:basic}}
\nomenclature{$\lambda$}{The composite function $f'\circ u$, Notation~\ref{N:basic}}

\begin{notation}
\label{N:basic}
We can assume below that $u\in C_{\rc}(\R^{+}\times\R)$, because our considerations are local in space-time.
We adopt the short notation $\lambda(t,x)=f'(u(t,x))$ for the charactersitic speed.
\end{notation}

\begin{notation}
\label{N:restrictions}
Given a function of two variables $\varphi(t,x)$, one denotes the restrictions to coordinate sections as
\[
\varphi_{x}^{\reu}(t):\ t\mapsto \varphi(t,x) 
\qquad
\varphi_{t}^{\red}(x):\ x\mapsto \varphi(t,x)
.
\]
\end{notation}
\nomenclature{$\varphi_{x}^{\reu}(t)$}{Restriction of a function $\varphi(t,x)$ to the first coordinate, Notation~\ref{N:restrictions}}
\nomenclature{$\varphi_{t}^{\red}(x)$}{Restriction of a function $\varphi(t,x)$ to the second coordinate, Notation~\ref{N:restrictions}}

\begin{notation}
\label{N:derivatives}
Given a function of locally bounded variation $\varphi(t,x)$, one denotes by
\[
\Dt \varphi(dt,dx)\qquad
\Dx\varphi(dt,dx)
\]
the measures of its partial derivatives. When it is not known if they are measures, we rather denote the distributional partial derivatives by
\[
\pt \varphi(t,x) ,
\qquad
\px\varphi(t,x).
\]  
Classical partial derivatives are often denoted by
\[
\pat{\varphi(t,x)} ,
\qquad
\pax{\varphi(t,x)}.
\]
\end{notation}
\nomenclature{$\Dt, \Dx$}{Partial derivatives of a function of bounded variation, Notation~\ref{N:derivatives}}
\nomenclature{$\pt,\px$}{Distributional partial derivatives, Notation~\ref{N:derivatives}}
\nomenclature{$\pat{},\pax{}$}{Classical partial derivatives, Notation~\ref{N:derivatives}}

\begin{definition}[Characteristic Curves]
\label{D:charactcurves}
\emph{Characteristic curves} of $u$ are absolutely continuous functions $\gamma:\R^{+}\to\R$, or equivalently the corresponding curves $i_{\gamma}:=(\Id,\gamma)$, defined on a connected open subset of $\R$ and satisfying the ordinary differential equation
\[
\dot \gamma(s) = \lambda(s,\gamma(s))=\lambda(i_{\gamma}(s)).
\]
The continuity of $u$ implies that $\gamma$ is continuously differentiable.
\end{definition}
\nomenclature{$\gamma, i_{\gamma}$}{Characteristic curve, Definition~\ref{D:charactcurves}}
Notice that $i_{\gamma}$ is an integral curve of the vector field $(1,\lambda)$.

\begin{figure}[ht]\centering
\includegraphics[width=.5\linewidth]{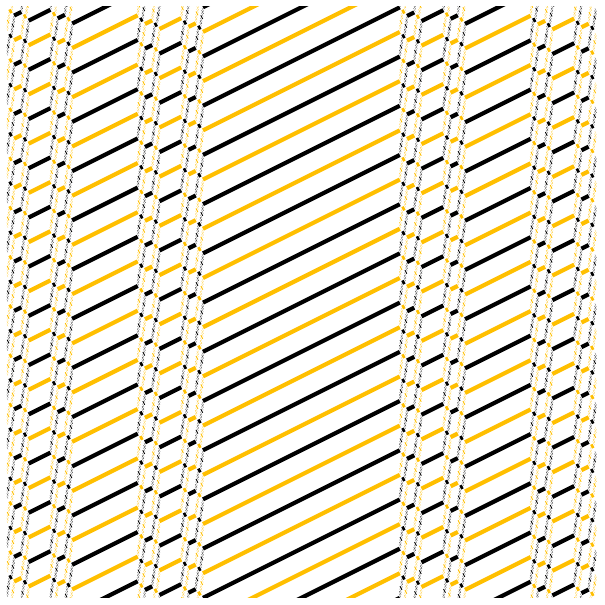}
\caption{Curves satisfying $\dot \gamma(s)=1$ almost everywhere may fail to be Lipschitz continuous, this is why characteristic curves are required to be absolutely continuous. Being automatically $C^{1}$-functions, they are then stable under uniform convergence.}
\label{fig:charactNonLip}
\end{figure}

\begin{definition}[Lagrangian Parameterization]
\label{D:LagrangianParameterization}
We call \emph{Lagrangian parameterization} associated with $u$ a surjective continuous function $\chi:\R^{+}\times\R\to\R$, or equivalently $\chi:\R\to C(\R^{+})$, such that\footnote{There is no reason for asking the following condition only for $\Ll^{1}$-a.e.~$y$: if it holds for $\Ll^{1}$-a.e.~$y$ then it holds naturally for every $y$. As well, it would be odd requiring the second condition only $\Ll^{1}$-a.e.~$t$.}
\begin{itemize}
\item[-] for each $y\in\R$, the curve $\chi(y)$ defined by $t\mapsto\chi(t,y)=\chi_{y}^{\reu}(t)$ is a characteristic curve:
\[
\dot \chi_{y}^{\reu}(t)=\pt\chi(t,y)=f'(u(t,\chi(t,y)))=\lambda(t,\chi(t,y))=\lambda(i_{\chi(y)}(t));
\]
\item[-] for each $t\in\R^{+}$, $y\mapsto\chi(t,y)=\chi^{\red}_{t}(y)$ is nondecreasing.
\end{itemize}
\end{definition}
\nomenclature{$\chi$}{Lagrangian parameterization for a continuous solution $u$ to \eqref{E:basicPDE}, Defintion~\ref{D:LagrangianParameterization}}


\begin{definition}\label{D:LagrangianParameterizationabsolutelycontinuous}
We call a Lagrangian parameterization $\chi$ \emph{absolutely continuous} if $(i_{\chi}^{-1})_{\sharp}\Ll^{2}\ll \Ll^{2}$.
Equivalently, $\Ll^{2}$-positive measure sets can not be negligible along the characteristics of the parameterization $\chi$: $\chi$ maps negligible sets into negligible sets.
\end{definition}

\begin{remark}
Even if $\chi$ generally is not injective, $(i_{\chi}^{-1})_{\sharp}\Ll^{2}$ is as well a well defined Borel measure meaning that for all compact subsets $K\subset \R^{+}\times\R$ one defines $(i_{\chi}^{-1})_{\sharp}\Ll^{2}:=\Ll^{2}(i_{\chi}(K))$. Of course $\left((i_{\chi}^{-1})_{\sharp}\Ll^{2}\right)(\emptyset)=0$ but one has also that if $A$ and $B$ are disjoint compact subsets of the plane then for all $t$ the intersection of $\chi\left(A\cap\{t\}\right)$ and $\chi \left(B\cap\{t\}\right)$ is at most countable, due to the monotonicity of $\chi^{\red}_{t} $:
\[
\Ll^{1}\Big( \chi\left(A\cap\{t\}\right)\bigcap \chi \left(B\cap\{t\}\right)\Big)=0\quad \forall t
\qquad
\Rightarrow
\qquad
\Ll^{2}\left(i_{\chi}(A)\cap i_{\chi}(B)\right)=0
\]
In particular
\[
\Ll^{2}\left(i_{\chi}(A)\right)+\Ll^{2}\left( i_{\chi}(B)\right)= \Ll^{2}\left(i_{\chi}(A)\cup i_{\chi}(B)\right)=\Ll^{2}\left(i_{\chi}(A\cup B)\right). 
\]
This implies that $(i_{\chi}^{-1})_{\sharp}\Ll^{2}$ is countably-additive, and thus a measure. This justifies our notation.
\end{remark}

\begin{notation}
\label{N:variousnotations}
We fix the following nomenclature, that we extend at the end of the paper.

\newcommand{\secondcolumn}{.8\textwidth}

\begin{tabbing}
\hspace*{0.5cm}\=$rrrrrrrrr$ \=\kill
 \> $X, k$
 \> Usually: $X$ a subset $ \R^{+}\times\R$, $k\in \N\cup\{\infty\}$
 \\
 \> $\Omega$
 \> Open subset of $\R^{+}\times\R$. Usually it can also supposed to be bounded.
 \\   \> $C(X)$
 \> \parbox{\secondcolumn}{Continuous functions on $X$}
 \\   \> $C_{\mathrm b}(X)$
 \> \parbox{\secondcolumn}{Bounded continuous functions on $X$}
 \\   \> $C^{1/\alpha}(X)$
 \> \parbox{\secondcolumn}{$1/\alpha$-H\"older continuous functions on $X$, where $0<1/\alpha\leq1$}
 \\    \> $C^{k}_{(\rc)}(\Omega)$
 \> \parbox{\secondcolumn}{$k$-times continuously differentiable (compactly supported) functions on $\Omega$}
 \\   \> $C^{k,1/\alpha}(\Omega)$
 \> \parbox{\secondcolumn}{$k$-times continuously differentiable (compactly supported) functions on $\Omega$ with}
  \\   \> {}
 \> \parbox{\secondcolumn}{$k$-th derivative which is $1/\alpha$-H\"older continuous in $\Omega$, where $0<1/\alpha\leq1$}
 \\   \> $\pw(X)$
 \> \parbox{\secondcolumn}{Borel bounded functions $\gfr$ on $X$}
 \\   \> $\pwu(X)$
 \> \parbox{\secondcolumn}{Equivalence classes $\gu$ made of those bounded Borel functions which coincide}
\\   \> {}
 \> \parbox{\secondcolumn}{ $\Ll^{1}$-a.e.~when restricted to \emph{any} characteristic curve of $u$}
 \\   \> $\pwx(X)$
 \> \parbox{\secondcolumn}{Equivalence classes $\gx$ made of those bounded Borel functions which coincide}
  \\   \> {}
 \> \parbox{\secondcolumn}{$\Ll^{1}$-a.e.~when restricted to $\{i_{\chi(y)}(t)\}_{t>0}$, \emph{for every} $y\in\R$ and for a fixed}
  \\   \> {}
 \> \parbox{\secondcolumn}{Lagrangian parameterization $\chi$}
 \\  
  \> $\distrBdd(X)$
 \> \parbox{\secondcolumn}{Equivalence classes $\gd$ of Borel bounded functions which coincide $\Ll^{2}$-a.e.}
 \\   \> $\D(\Omega)$
 \> \parbox{\secondcolumn}{Distributions on $\Omega$}
 \\   \> $\mathcal M(X)$
 \> \parbox{\secondcolumn}{Radon measures on $X$}
\end{tabbing}
\nomenclature{$X$}{Subset of $\R^{+}\times\R$, usually Borel.}
\nomenclature{$\Omega$}{Open subset of $\R^{+}\times\R$, if needed connected}
\nomenclature{$C(\Omega)$}{Continuous functions on $\Omega$, see also $C_{b}, C^{k}_{}, C^{k}_{\rc}, C^{k, 1/\alpha}$ in Notation~\ref{N:variousnotations}}
\nomenclature{$\distrBdd(X)$}{Bounded functions on $X$ identified $\Ll^{2}$-a.e., Notation~\ref{N:variousnotations}}
\nomenclature{$\D(\Omega)$}{Distributions on $\Omega$, Notation~\ref{N:variousnotations}}
\nomenclature{$\mathcal M(X)$}{Radon measures on $X$, Notation~\ref{N:variousnotations}}
\nomenclature{$\pw(X)$}{Functions defined pointwise on $X$, Notation~\ref{N:variousnotations}}
\nomenclature{$\pwu(X)$}{Functions coinciding $\Ll^{1}$-a.e.~on characteristics of $u$, Notation~\ref{N:variousnotations}}
\nomenclature{$\pwx(X)$}{Functions coinciding $\Ll^{1}$-a.e.~on the Lagrangian parameterization $\chi$, Notation~\ref{N:variousnotations}}
\nomenclature{$\gsimple,\gd$}{Distributional, bounded source term for the balance law~\eqref{E:basicPDE}}
\nomenclature{$\gfr$, $\gu$, $\gx$}{Functions beloning to $\pw(X)$, $\pwu(X)$, $\pwx(X)$ respectively, Notation~\ref{N:variousnotations}}
\nomenclature{$[\cdot]_{\lambda}$, $[\cdot]_{\chi}$, $[\cdot]$}{Projections on, $\pwu(X)$, $\pwx(X)$, $\distrBdd(X)$ respectively, Notation~\ref{notation:projection}}

\end{notation}

\begin{notation}
\label{notation:projection}
Notice there are the following natural correspondences
\begin{align*}
&\pw(X)&\xrightarrow{[\cdot]_{\lambda}}&&\pwu(X)&&\xrightarrow{[\cdot]_{\chi}}&&\pwx(X) \\
&\gfr &\mapsto&&\gu=[\gfr]_{\lambda}&&\mapsto&&\gx=[\gu]_{\chi}=[\gfr]_{\chi}.
\end{align*}
and moreover
\begin{align*}
&\pw(X)&&\xrightarrow{[\cdot]}&&\distrBdd(X)\\
&\gfr &&\mapsto&&\gd =[\gfr].
\end{align*}
The same brackets denote also correspondences from any of the bigger spaces: brackets identify the target spaces.
The correspondences among $\pwx(X),\pwu(X)$ and $\distrBdd(X)$ do not exist in general.
Trivially, sets which are $\Ll^{2}$-negligible generally are not $\Ll^{1}$-negligible along any characteristic curve $\gamma$ of~\eqref{E:basicPDE}.
Moreover, there exists a subset of the plane which has positive Lebesgue measure but which intersects each characteristic curve of a Lagrangian parameterization in a single point. See \cite[\S~4.1-2]{file2ABC}. A correspondence exists with \emph{absolute continuity}.
\begin{lemma}
\label{L:inclusioneinversa}
If a Lagrangian parameterization $\chi$ is absolutely continuous, for every 
Borel functions $\gfr_{1},\gfr_{2}\in\pw(X)$ such that $[\gfr_{1}]_{\chi}=[\gfr_{2}]_{\chi}$ one has that $[\gfr_{1}]=[\gfr_{2}]\in L^{\infty}(X)$.
\end{lemma}
\begin{proof}
It is an algebraic application of the definitions of the spaces in Notation~\ref{N:variousnotations}.
\end{proof}

\end{notation}

\begin{definition}
Let $\gd\in\distrBdd(\Omega)$.
If $u\in C(\Omega)$ satisfies
\[
\forall \varphi\in C^{\infty}_{\rc}(\Omega)
\qquad
\iint_{\Omega} \left\{\varphi_{t} \,u + \varphi_{x} \,f(u) \right\}= \iint_{\Omega} \varphi\, \gd
\]
we say that $u$ is a \emph{continuous distributional (or Eulerian)} solution of~\eqref{E:basicPDE}.
\end{definition} 

\begin{definition}[Lagrangian solution]
\label{D:lagrSol}
A function $u\in C(\Omega)$ is called \emph{continuos Lagrangian solution} of~\eqref{E:basicPDE} with Lagrangian parameterization $\chi$, associated with $u$, and Lagrangian source term $\gx\in\pwx(\Omega)$ if
\[
\text{for all $y\in \R$}\qquad
\ddt u(t,\chi(t,y) ) = \gx(t, \chi(t,y))
\qquad\text{in $\D\left(i_{\chi(y)}^{-1}(\Omega)\right)$}.
\]
\end{definition}

\begin{definition}[Broad solution]
Let $u\in C(\Omega)$ and $\gu\in\pwu(\Omega)$. The function $u$ is called \emph{continuous broad solution} of~\eqref{E:basicPDE} if it satisfies
\[
\text{for all characteristic curves $\gamma$ of $u$}\qquad
\ddt u(t,\gamma(t) ) = \gu(t, \gamma(t))
\qquad\text{in $\D\left(i_{\gamma}^{-1}(\Omega)\right)$.}
\]
\end{definition}

\begin{definition}
A continuous function $u$ is both a distributional/Lagrangian/broad solution of~\eqref{E:basicPDE} when the source terms are \emph{compatible}: if there exists a Borel function $\gfr$ such that
\[
\gd=[\gfr],\  \gx=[\gfr]_{\chi}, \ \gu = [\gfr]_{\lambda}.
\] 
\end{definition}

\begin{definition}
\label{D:inf{}lf}
We define
$z^{*}\in\R$ \emph{inf{}lection point} of a function $f\in C^{2}(\R)$ if $f''(z^{*})=0$ but it is neither a local maximum nor a local minimum for $ f(z)-f'(z^{*})(z-z^{*})$.
We denote by $\infl(f)$ the set of inf{}lection points of $f$, $\clos({\infl(f)})$ is its closure.
\end{definition}

In principle, $u$ could be a distributional solution of~\eqref{E:basicPDE} with source $\gd$ and a Lagrangian solution with source $\gx$ with $\gd$ and $\gx$ which do not correspond to a same function $\gfr\in\pw(\R^{+}\times\R)$: in this case, we would not say that $u$ is both a distributional and Lagrangian solution to the \emph{same} equation, because source terms are different. We discuss the issue in \cite{file2ABC}, where we prove that if the inf{}lection points of $f$ are negligible then whenever a same function is a Lagrangian solution and it is a distributional solution then the source terms must be compatible. 

\subsection{Overview of the results}

Now that definitions are clear, we describe our results:
\begin{itemize}
\item[\S~\ref{Ss:auxiliaryIntro}:] we collect observations on Lagrangian parameterization and Lagrangian/Broad solution;
\item[\S~\ref{Ss:mainIntro}:] we summarize relations among the different notions of solutions of~\eqref{E:basicPDE}.
\end{itemize}
Notice first that both the definitions above and the statements below are local in space-time, as well as the compatibility of the sources that will be discussed in~\cite{file2ABC}. This motivates the assumption that $u$ is compactly supported, that we fixed in Notation~\ref{N:basic}. 

\subsubsection{Auxiliary observations}
\label{Ss:auxiliaryIntro}
We begin collecting elementary observations on the basic concept of Lagrangian parameterization and Lagrangian/Broad solution, mostly for consistency.

\begin{lemma}
\label{L:Elagrpar}
There exists a Lagrangian parameterization $\chi$ associated with any $u\in C_{b}(\Omega)$.
In particular, one has the implication
\[
\text{continuous broad solution}
\quad
\Rightarrow
\quad
\text{continuous Lagrangian solution, with $\gx=[\gu]_{\chi}$.}
\]
The converse implication holds under a condition on the inf{}lection points of $f$, not in general.
\end{lemma}
\begin{proof}
An explicit construction of a Lagrangian parameterization $\chi$ is part of \S~\ref{Ss:lagrparam}.
It relies on Peano's existence theorem for ODEs with continuous coefficients.
If $\gu\in\pwu(\Omega)$ is the Broad source, then $[\gu]_{\chi}$ is immediately the Lagrangian source.
The converse implication does not always hold, see~\cite[\S~4.3]{file2ABC} .
\end{proof}

\begin{lemma}
\label{L:lipcharisLagr}
Let $u\in C(\clos\Omega)$ and $G>0$.
Assume that through every point of a dense subset of $\Omega$ there exists a characteristic curve along which $u$ is $G$-Lipschitz continuous.
Then there exists a Lagrangian parameterization $\chi$ along whose characteristics $u$ is $G$-Lipschitz continuous.
\end{lemma}
\begin{proof}
The proof is given in  \S~\ref{Ss:lagrparam}.
\end{proof}

\begin{lemma}
\label{L:hogdipdatx}
Let $u\in C(\clos(\Omega))$ and $G>0$. 
A sufficient condition for $u$ being a Lagrangian solution of~\eqref{E:basicPDE}
is the existence of a Lagrangian parameterization $\chi$ such that
\[
\text{for all $y$ the distribution $\ddt u(t,\chi(t,y) )$ is uniformly bounded in $\D\left(i_{\chi(y)}^{-1}(\Omega)\right)$.}
\]
\end{lemma}
\begin{proof}
The proof is given in \S~\ref{Ss:lipalongchar}.
\end{proof}

\nomenclature{$\infl(f)$}{Inf{}lection points of $f$, Definition~\ref{D:inf{}lf}}
\nomenclature{$\clos(\cdot)$}{Closure of a set}

\subsubsection{Main results}
\label{Ss:mainIntro}
In the present paper we do not discuss existence of continuous solutions of~\eqref{E:basicPDE}, but we assume that we are given a continuous function $u$: due to the lack of regularity, the focus of this paper is in which sense it can be a solution of the PDE~\eqref{E:basicPDE}.

We first state one of the important conditions: we denote by \eqref{hyp:H} the assumption 
\[
\tag{H}
\text{The set of inf{}lection points $\clos({\infl(f)})$ of Definition~\ref{D:inf{}lf} is $\Ll^{1}$-negligible.}
\label{hyp:H}
\]
We \emph{roughly} summarize our results with the following implications:

{\centering\begin{tabular}{ccccccccc}
Broad 
&& \begin{tabular}{c} $\Longrightarrow$ always, Th.~\ref{L:Elagrpar}\\ \hline $\Longleftarrow$ if \eqref{hyp:H} holds, \S~\ref{S:distributionaltoBroad}\end{tabular}
&& Lagrangian
&&${}^{\S~\ref{S:distrareLagrangian}}\Longleftrightarrow{}^{\S~\ref{S:lagraredistr}}$
&& distributional \\
\end{tabular}

}

The distinction among Lagrangian and distributional continuous solutions is motivated by the fact that the two formulations are different, and it is not that trivial proving their equivalence. Moreover, Lagrangian and distributional source terms do not correspond automatically, as we discuss in~\cite{file2ABC}. In particular, if we do not assume the negligibility of inf{}lection points we are not yet able to say that the Lagrangian and distributional source terms must be compatible.
If the f{}lux function is for example analytic, then our work gives instead a full analysis.

We collect also in the table below interesting properties of the solution. The properties depend on general assumptions on the smooth flux function $f$: 
\begin{enumerate}
\item whether $f$ satisfies a convexity assumption named in~\cite{file2ABC} $\alpha$-convexity, $\alpha>1$, which for $\alpha=2$ is the classical uniform convexity;
\item whether the closure of inflection points of $f$ is negligible, as defined in~\eqref{hyp:H} above.
\end{enumerate}

\vskip\baselineskip
\begin{tabular}{p{5.2cm}|c|c|c}
&$\alpha$-convexity & Negligible inf{}lections & General case\\
\hline
absolutely continuous La\-gran\-gian pa\-ra\-me\-terization & \ding{55} \cite[\S~4.1
]{file2ABC} & \ding{55} & \ding{55}\\
\hline
$u$ H\"older continuous & \ding{51} \cite[\S~2.1]{file2ABC}&\ding{55} \cite[\S~4.2
]{file2ABC}& \ding{55}\\
\hline
$u$ $\Ll^{2}$-a.e.~differentiable along characteristic curves &\ding{51}\cite[\S~2.2]{file2ABC}&\ding{55} \cite[\S~4.2
]{file2ABC}& \ding{55}\\
\hline
$u$ Lipschitz continuous along characteristic curves & \ding{51}& \ding{51} \ Theorem~\ref{T:sharpLipschitzreg}& \ding{55} \cite[\S~4.3
]{file2ABC}\\
\hline
entropy equality  & \ding{51} & \ding{51} & \ding{51}Lemma~\ref{L:nodissipation2}\\
\hline
compatibility of sources  & \ding{51} \ding{51} \cite[\S~2.2]{file2ABC}& \ding{51} \cite[\S~3]{file2ABC}& 
\end{tabular}
\vskip\baselineskip
{We show in Corollary~\ref{L:BVcontsol1} that if the continuous solution $u$ has bounded total variation then one can as well select a Lagrangian parameterization which is absolutely continuous, for $f\in C^{2}$.}

\section{Lagrangian solutions are distributional solutions}
\label{S:lagraredistr}

Consider a continuous Lagrangian solution $u(t,x)$ of~\eqref{E:basicPDE} in the sense of Definition~\ref{D:lagrSol}. Let $\chi$ be a Lagrangian parameterization, $\gx\in\pwx(\Omega)$ be its source term and set $G=\norm{\gx}_{\infty}$.
We want to show that there exists $\gd\in\distrBdd(\Omega)$ such that $u(t,x)$ is a distributional solution of
\begin{equation*}
\tag{\ref{E:basicPDE}}
\pt u(t,x) + \px (f(u(t,x))) =  \gd(t,x)
\qquad f\in C^{2}(\R),
\qquad 
|\gd(t,x)|\leq G\ .
\end{equation*}
We do not discuss at this stage the compatibility of the source terms $\gx$ and $\gd$.

\begin{notation}
\label{N:notationcompactsupp}
We already observed in the introduction that we are considering local statements. We directly assume therefore
\begin{itemize}
\item $\Omega=\R^{+}\times\R$,
\item $u$ compactly supported.
\end{itemize}
We set $\Lambda=\max \lambda = \max f'(u)$. We recall that we set $G=\norm{\gx}_{\infty}$.
\end{notation}

\subsection{The case of \texorpdfstring{$\BV$}{BV}-regularity}
\label{Ss:BVregolare}
In the present section we assume that $u$ is not only continuous but also that it has bounded variation.
Under this simplifying assumption, we prove in Lemma~\ref{L:BVcontsol2} below that $u$ is a distributional solution to~\eqref{E:basicPDE}, with the natural candidate for $\gd$.
The proof is based on explicit computations. Computations of this section exploit Vol'pert chain rule and the possibility to produce a change of variables which is absolutely continuous, as we state in Corollary~\ref{L:BVcontsol1} below.
It follows by the following more general lemma.

\begin{lemma}
\label{L:BVcontsol1astratto}
Consider a function $w:\R^{+}\times\R\to\R$ such that 
\begin{itemize}
\item the restriction $w_{\ytau}^{\reu}$ belongs to $C^{1,1}(\R^{+})$ for all $y\in\R$ and
\item the second mixed derivative $\partial_{t\ytau}w$ is a Radon measure.
\end{itemize}
Then, up to reparameterizing the $\ytau$-variable, there exists $0\leq H\in\distrBdd_{\loc}(\R^{+}\times\R)$ such that
\begin{align*}
&\Dytau w=H\Ll^{2},  &&\Dt H= \Dytau \left(\pat{ w}\right) \in \mathcal M(\R^{+}\times\R).
\end{align*}
 \end{lemma}

We rather prefer to prove the following corollary, which is more related to the notation we adopt: the proof of Lemma~\ref{L:BVcontsol1astratto} is entirely analogous.
The irrelevant disadvantage is that the commutation of the $t$- and $\ytau$- distributional derivatives is less evident than in the above lemma.
 
\begin{corollary}
\label{L:BVcontsol1}
Let $u$ be a continuous Lagrangian solution of~\eqref{E:basicPDE} such that $\partial_{x}u(t,x)$ is a Radon measure.
Then one can choose a Lagrangian parameterization $\chi$ which is absolutely continuous (see Definition~\ref{D:LagrangianParameterizationabsolutelycontinuous}): this additional regularity allows the injection
\begin{align*}
&\pwx(\Omega)&&\xrightarrow{[\cdot]}&&\distrBdd(\Omega)\\
&\gx=[\gfr]_{\chi} &&\mapsto&&\gd =[\gfr]=[\gx].
\end{align*}
Moreover, for every test function $\varPhi(t,\ytau)\in C^{1}_{\rc}(\R^{+}\times\R)$ and for $\Ll^{1}$-a.e.~$t$ one has
\begin{align}
\label{E:derivataperparti}
\ddt \int \varPhi (t,\ytau)H(t,\ytau)d\ytau   
&=\int \varPhi(t,\ytau) \Dif f'(U_{t}^{\reu}(d\ytau))+\int \pat{\varPhi(t,\ytau)} H(t,\ytau)d\ytau 
\end{align}
\end{corollary}
\begin{lemma}
\label{L:BVcontsol2}
Under the assumptions of Corollary~\ref{L:BVcontsol1}, $u$ has locally bounded variation. Moreover, denoting by $\gx=[\gfr]_{\chi}$ a Lagrangian source, then one has
\[
\Dt u(dt,dx) + \Dx f(u(dt,dx)) = \gfr(t,x)dtdx.
\]
\end{lemma}

Note that Lemma~\ref{L:BVcontsol1} does not follow from the theory of ODEs with rough coefficients because the notion of Lagrangian parameterization is more specific than a solution of a system of ODEs.
In this paper, where the focus is on the PDE~\eqref{E:basicPDE}, we rather prefer to prove the lemma in the form of Corollary~\ref{L:BVcontsol1}. We stress once more that computations below, switching notations, prove indeed Lemma~\ref{L:BVcontsol1astratto}, proof which could be slightly shortened in a more abstract setting.

\begin{remark}
\label{R:equivBVhyp}
Let $U(t,\ytau):=u(t,\chi(t,\ytau))$ for some Lagrangian parameterization $\chi$.
We notice that one can equivalently assume that either $\px u$ or $\pytau U(t,\ytau)$ is a Radon measure.
This is a direct consequence of the slicing theory of $\BV$ functions, because $\chi_{t}^{\red}(\ytau)$ is monotone and the total variation of $u_{t}^{\red}(x)$ is equal to the total variation of $U_{t}^{\red}(\ytau)$. 
\end{remark}

\begin{proof}[Proof of Lemma~\ref{L:BVcontsol1}]Let $\chi$ be a Lagrangian parameterization corresponding to $u$.
By assumption and by Remark~\ref{R:equivBVhyp}, for $\Ll^{1}$-a.e.~$t$ also the function 
\[\ytau\mapsto u(i_{\chi(\ytau)}(t))=:U_{t}^{\reu} (\ytau)\]
has locally bounded variation.
Moreover, for every $\ytau$ by definition of Lagrangian solution
\[
t \mapsto u(i_{\chi(\ytau)}(t))=:U^{\red}_{\ytau}(t)
\]
is Lipschitz continuous. We deduce by the slicing theory of $\BV$-functions~\cite[Th.~3.103]{AFP} that also the function $(t,\ytau)\mapsto U(t,\ytau)$ has locally bounded variation.
We show now that the Lagrangian parameterization $\chi$ can be here assumed to be absolutely continuous.

\step{Renormalization of $y$ for absolutely continuity of $\chi$}
Consider the two coordinate disintegrations of the measure on the plane given by $\Dytau U(t,\ytau)$:
by the classical disintegration theorem~\cite[Th.~2.28]{AFP} there exists a nonnegative Borel measure $m\mathcal \in \mathcal M^{+}(\R)$ and a measurable measure-valued map $\ytau\mapsto\nu_{\ytau}\in\mathcal M^{+}(\R)$ such that
\begin{equation}
\label{E:disintdtauu}
\Dytau U(dt,d\ytau)
=
\int
{\Dif U_{t}^{\reu}(d\ytau)} dt = \int \nu_{\ytau}(dt) m(d\ytau) .
\end{equation}
The first equality is just the slicing theory for $\BV$ functions~\cite[Th.~3.107]{AFP}.

\begin{claim}
\label{C:abscontparam}
Consider the Lagrangian parameterization
$\bar \chi(t,y):= \chi(t,h^{-1}(y))$ with $h$ defined by
\begin{align*}
h(\ytau)&:=\ytau+m((-\infty,\ytau])+\Dif\chi^{\red}_{0}((-\infty,\ytau]).
\end{align*}
Then one has that $h_{\sharp}m\ll \Ll^{1}$ and $h_{\sharp}\Dif\chi^{\red}_{0}\ll\Ll^{1}$ with densities bounded by $1$.
\end{claim}

\begin{proof}[Proof of Claim~\ref{C:abscontparam}]
Fix any $a\leq b$.
We first observe that
\begin{align*}
h(b)-h(a)&=b-a+m((a,b])+\Dif\chi^{\red}_{0}((a,b]) \\
&\geq \max\{b-a, m((a,b]), \Dif\chi^{\red}_{0}((a,b])\}\geq 0.
\end{align*}
This shows that $\Dif h\geq \Ll^{1}$, $\Dif h\geq m$ and $\Dif h\geq\Dif\chi^{\red}_{0}$.
Since $U$ and $\chi$ are continuous functions, then $\Dytau U$ and $\Dif\chi^{\red}_{0}$ are continuous measures and therefore $h$ is a continuous function.
Fix any $a<b$ and suppose $h(a')= a $, $h(b')=b$. Then one verifies that
\begin{align*}
\frac{h_{\sharp}m([a,b])}{b-a}&\leq\frac{m([ a', b'])}{h( b')-h( a')} \\
&\leq \frac{m([ a', b'])}{b'-a'+m((a',b'])+\Dif\chi^{\red}_{0}((a',b'])}<1.
\end{align*}
This concludes the proof for $h_{\sharp}m$, and $h_{\sharp}\Dif\chi^{\red}_{0}$ is entirely similar.
\end{proof}

Claim~\ref{C:abscontparam} assures that one can reparameterize the $\ytau$-variable so that both $m(d\ytau)$ and $\Dif\chi^{\red}_{0}(\ytau)$ are absolutely continuous with bounded densities. 
Let $\vartheta(\ytau)$ and $\beta(\ytau)$ be their Radon-Nicodym derivatives w.r.t. $\Ll^{1}$ after, eventually, the reparameterization of $y$:
\begin{equation}
\label{E:alphabeta}
m(d\ytau)=:\vartheta(\ytau)d\ytau,
\qquad
\Dif\chi^{\red}_{0}(\ytau)=:\beta(\ytau)d\ytau \qquad \vartheta,\beta\in L^{\infty}.
\end{equation}

\step{Formula for $\Dytau \chi$}
In this step we prove Claim~\ref{C:expressiondychi} below, which implies that both the distributional partial derivatives of $\chi$ are absolutely continuous measures. The claim below yields then that $\chi$ is an absolutely continuous Lagrangian parameterization.
As  $\chi(t,\ytau)$ maps negligible sets into negligible sets, then one has the inclusion stated in Lemma~\ref{L:inclusioneinversa}:
\begin{align*}
&
\pwx(\R^{+}\times\R)
&\hookrightarrow
&&\distrBdd(\R^{+}\times\R).
\end{align*}

\begin{claim}
\label{C:expressiondychi}
The measure $\Dytau \chi(dt,d\ytau)$ is given by the following formula:
\begin{subequations}
\label{E:disintdtauchi}
\begin{gather}
\label{E:defH1}
\Dytau\chi(dt,d\ytau)=H(t,\ytau) dtd\ytau, \qquad \distrBdd_{\loc}\left(i^{-1}_{\chi}(\Omega)\right)\ni H(t,\ytau)
\geq0
\\
\label{E:defH}
H(t,\ytau) =\vartheta(\ytau)\int^{t}_{0} f''(U(s,\ytau))\nu_{\ytau}(ds)+\beta(\ytau) .
\end{gather}
\end{subequations}
\end{claim}

\begin{proof}[Proof of Claim~\ref{C:expressiondychi}]
Being
\begin{equation}
\label{E:dotchiformula}
\dot \chi_{\ytau}^{\reu}(t) = f'(U(t,\ytau))
\end{equation}
by Vol'pert chain rule~\cite[Th.~3.96]{AFP} and~\eqref{E:disintdtauu} one has the following disintegration
\begin{equation}
\label{E:disintdotchi}
\Dytau f'(U(dt,d\ytau)) = f''(U(t,\ytau)) \Dytau U(dt,d\ytau) 
= \int \left\{f''(U(t,\ytau))\nu_{\ytau}(dt)\right\} m(d\ytau)
.
\end{equation}
One can compute $\Dytau \chi$ in the following way: write $\chi$ as a primitive and differentiate under the integral.
For every test function $\varPhi(t,\ytau)$
\begin{align*}
-\iint \varPhi(t,\ytau) \Dytau\chi(dt,d\ytau)
&=
\iint \paytau {\varPhi(t,\ytau)} \chi(t,\ytau)d\ytau dt
\\&=
\iint \paytau {\varPhi(t,\ytau)}\left\{ \int^{t}_{0}\dot\chi_{\ytau}^{\reu}(s)ds+\chi(0,\ytau)\right\}d\ytau dt
\\&=
 \iiint^{t}_{0}\paytau {\varPhi(t,\ytau)}\left[\dot\chi_{\ytau}^{\reu}(s)+\chi^{\red}_{0}(\ytau)/t\right]dsd\ytau dt
\\
&\stackrel{\eqref{E:dotchiformula}}{=}
\int \left\{ \iint^{t}_{0}\paytau {\varPhi(t,\ytau)}\left[ f'(U(s,\ytau)) +\chi^{\red}_{0}(\ytau)/t\right]dsd\ytau\right\}dt
\\
&=
-\int \left\{ \iint^{t}_{0} {\varPhi(t,\ytau)}\left[ \Dytau f'(U(ds,d\ytau)) +\Dif\chi^{\red}_{0}(\ytau)/t\right]\right\}dt
\end{align*}
The last step was allowed because $U\in\BV$. Owing to~\eqref{E:disintdotchi},~\eqref{E:alphabeta} we can now proceed with
\begin{align*}
&{=}
-\int \left\{ \iint^{t}_{0} \varPhi(t,\ytau)f''(U(s,\ytau))\nu_{\ytau}(ds)\vartheta(\ytau)d\ytau\right\} dt- \int\left\{\iint^{t}_{0}\left[\beta(\ytau)/t\right]dsd\ytau \right\}dt 
\\
&\stackrel{}{=}
- \iint\left\{\varPhi(t,\ytau)\left[\beta(\ytau)+\vartheta(\ytau)\int^{t}_{0} f''(U(s,\ytau))\nu_{\ytau}(ds)\right]\right\}d\ytau dt .
\end{align*}
Note that $H$ is the function within the inner square brackets, thus we proved the claim.
\end{proof}

\step{Time derivative of $H$}
Definition~\eqref{E:defH} of $H$ does not directly allow to differentiate $H$ in the $t$ variable, because the measure $\nu_{\ytau}$ may not be absolutely continuous.
Nevertheless, this is possible from~\eqref{E:defH1}, obtaining that $\Dt H$ is a Radon measure.
\begin{claim}
For every test function $\varPhi(t,\ytau)\in C^{1}_{\rc}(\R^{+}\times\R)$ and for $\Ll^{1}$-a.e.~$t$ one has~\eqref{E:derivataperparti}
\begin{align}
\ddt \int \varPhi (t,\ytau)H(t,\ytau)d\ytau   
&=\int \varPhi(t,\ytau) \Dif f'(U_{t}^{\reu}(d\ytau))+\int \pat{\varPhi(t,\ytau)} H(t,\ytau)d\ytau 
\end{align}
\end{claim}

\begin{proof}[Proof of Claim~\ref{E:derivataperparti}]
Consider the limit of the incremental ratios.
Integrate by parts in $\ytau$ before the limit, take then the limit in $h$ and integrate by parts again in $y$. By the 
weak continuity of
\[
t\mapsto \Dif f'(U_{t}^{\reu}(d\ytau))
\]
one has
\begin{subequations}
\label{E:accessssorie}
\begin{align}
\int \varPhi(t,\ytau) \Dif f'(U_{t}^{\reu}(d\ytau))&=\lim_{h\downarrow0}\frac{1}{h}\left\{\int \int_{t}^{t+h}\varPhi(t,\ytau)\Dytau f'(U(ds,d\ytau))\right\}.
\end{align}
Remembering \eqref{E:alphabeta}, \eqref{E:disintdotchi} and then the definition~\eqref{E:defH} of $H$ one has
\begin{align}
\int \int_{t}^{t+h}\varPhi(t,\ytau)\Dytau f'(U(ds,d\ytau))
&=\int \varPhi(t,\ytau)\left[\int_{t}^{t+h}f''(U(s,\ytau))\nu_{\ytau}(ds)\vartheta(\ytau)\right]d\ytau
\\&=\int \varPhi(t,\ytau)\left[H(t+h,\ytau)-H(t,\ytau)\right]d\ytau.
\end{align}
\end{subequations}
Owing to~\eqref{E:accessssorie} one can deduce that for $\Ll^{1}$-a.e.~$t$ equation~\eqref{E:derivataperparti} holds:%
\begin{align*}
\ddt& \int \varPhi (t,\ytau)H(t,\ytau)d\ytau =  \lim_{h\to 0}\frac{1}{h}\left\{\int \varPhi (t+h,\ytau)H(t+h,\ytau)d\ytau -\int \varPhi (t,\ytau)H(t,\ytau)d\ytau\right\}  \\
&=\lim_{h\to 0}\frac{1}{h}\left\{\int \varPhi(t,\ytau)\left[H(t+h,\ytau)-H(t,\ytau)\right]d\ytau\right\}\\
&\qquad\qquad+\int \lim_{h\to 0}\frac{\varPhi(t+h,\ytau)-\varPhi(t,\ytau)}{h}H(t+h,\ytau)d\ytau
\\
&\stackrel{\eqref{E:accessssorie}}{=}\int \varPhi(t,\ytau) \Dif f'(U_{t}^{\reu}(d\ytau))+\int \pat{\varPhi(t,\ytau)} H(t,\ytau)d\ytau .\qedhere
\end{align*}
\end{proof}
The proof of the absolute continuity of suitable Lagrangian parameterizations is ended.
\end{proof}

\begin{proof}[Proof of Lemma~\ref{L:BVcontsol2}]
We now prove that the PDE~\eqref{E:basicPDE} holds in distributional sense.
When $\px u$ is a Radon measure, this implies by Vol'pert chain rule that $u$ has locally bounded variation.
For every test function $\varphi\in C^{\infty}_{\rc}(\Omega)$, one can apply in the integral
\[
\langle   \pt u+ \px f(u), \varphi \rangle
=
-\iint_{\Omega} \left\{ \pat{\varphi(t,x)}  u(t,x) + \pax{ \varphi(t,x)}f(u(t,x)) \right\}dtdx
\] 
the following change of variables, that one can assume absolutely continuous by Corollary~\ref{L:BVcontsol1}:
\begin{equation}
\label{E:chgvar}
\Psi\ :\ 
\begin{pmatrix}
t\\
\ytau
\end{pmatrix}
\mapsto
\begin{pmatrix}
t\\
x
\end{pmatrix}
:=
\begin{pmatrix}
t\\
\chi(t,\ytau)
\end{pmatrix}
.
\end{equation}
Denote $\varPhi(t,\ytau)=\varphi(i_{\chi(\ytau)}(t))$ and $U(t,\ytau)=u(i_{\chi(\ytau)}(t))$.
Remembering~\eqref{E:dotchiformula} one obtains
\begin{equation}
\label{E:chvarmon1}
\begin{split}
\langle  \pt u+ \px f(u), \varphi \rangle
=&
-\iint \pat{\varPhi(t,\ytau)}  U(t,\ytau) \Dytau\chi(dt,d\ytau)
\\
+\iint &f'(u(i_{\chi(\ytau)}(t)))\paytau{\varPhi(t,\ytau)}  U(t,\ytau) dtd\ytau
-\iint\paytau{ \varPhi(t,\ytau)}f(U(t,\ytau)) dtd\ytau.
\end{split}
\end{equation}

\firststep
\step{$y$-derivatives}
The last two addends in~\eqref{E:chvarmon1}, integrating by parts, are just
\begin{equation}
\label{E:lasttwoadd}
\iint \varPhi(t,\ytau)\Dytau\big[f(U)-f'(U)U\big](dt,d\ytau) .
\end{equation}
Notice that $f'(U)U$ is still a function with locally bounded variation, and that its derivative can be computed by Vol'pert chain rule: it is equal to
\[
\Dytau \left[ U f'(U) \right] = \left[f'(U)+Uf''(U)\right] \Dytau U = \Dytau f(U)+ U\Dytau f'(U).
\]
After simplifying the first term in~\eqref{E:lasttwoadd}, therefore, we find that the last two addends in~\eqref{E:chvarmon1} are
\begin{equation}
\label{E:lasttwoadd2}
-\iint \varPhi(t,\ytau) U(t,\ytau) \Dytau f'(U(dt,d\ytau)) .
\end{equation}

\step{$t$-derivative}
The first addend in~\eqref{E:chvarmon1} is more complex and requires the properties of $H$ in~\eqref{E:disintdtauchi},~\eqref{E:derivataperparti}.
Notice that $\Phi(t,y)U(t,y)$ is absolutely continuous in time. By the additional regularly in~\eqref{E:derivataperparti} of $H$ one has the integration by parts
\begin{equation}
\label{E:chvarmon32}
\begin{split}
-\iint \pat{\varPhi(t,\ytau)}  U(t,\ytau) \Dytau\chi(dt,d\ytau)
=-\iint \pat{\varPhi(t,\ytau)}  U(t,\ytau) H(t,\ytau)dtd\ytau
=
\\
\iint {\varPhi(t,\ytau)} \pat{U(t,\ytau)} \Dytau\chi(dt,d\ytau)
+\iint {\varPhi(t,\ytau)}  U(t,\ytau) \Dytau f'(U(dt,d\ytau))
\end{split}
\end{equation}
Thanks to the absolute continuity of $\chi$ that one can assume by Corollary~\ref{L:BVcontsol1}, the term $\pat{U(t,\ytau)}$ in the first integral in the RHS of~\eqref{E:chvarmon32} is just the Lagrangian source term $\gx=[\gfr]_{\chi}$ evaluated at $i_{\chi(\ytau)}(t)$. The first addend in the RHS of~\eqref{E:chvarmon32} can be thus rewritten just as
\[
\iint\varphi(t,x)\gfr(t,x)dtdx.
\]
The remining addend in the RHS of~\eqref{E:chvarmon32} instead cancels the two remaining terms in~\eqref{E:chvarmon1}, by their equivalent form~\eqref{E:lasttwoadd2}.
After the cancellation we find that~\eqref{E:chvarmon1} is just
\[
\iint \varphi(t,x) \left[ \Dt u+ \Dx f(u)\right](dt,dx)
=
\iint\varphi(t,x)\gx(t,x)dtdx.
\qedhere
\]
\end{proof}

\subsection{The case of continuous solutions: \texorpdfstring{$\BV$}{BV} approximations}
We provide in this section the proof that continuous Lagrangian solutions of the balance law~\eqref{E:basicPDE} are also distributional solutions, without assuming $\BV$-regualrity.
In order to prove it, we construct a sequence of approximations having bounded variation, so that we take advantage of \S~\ref{Ss:BVregolare}.
We omit here the correspondence of the source terms, discussed separately in \cite{file2ABC}.

\begin{lemma}
\label{L:monotoneApproxiamtion}
Let $u$ be a continuous Lagrangian solution of~\eqref{E:basicPDE} with source term bounded by $G$. 
Then there exists a sequence of continuous functions $u_{k}(t,x)$, ${k\in\N}$, which are both
\begin{itemize}
\item functions of bounded variation;
\item Lagrangian and Eulerian solutions of~\eqref{E:basicPDE} with source terms bounded by $G$;
\item converging uniformly to $u$ as $k\uparrow\infty$.
\end{itemize}
\end{lemma}

\begin{corollary}
\label{C:lagrnodiss}
Let $u$ be a continuous Lagrangian solution of~\eqref{E:basicPDE}. 
Then it is a continuous distributional solution of~\eqref{E:basicPDE} and it does not dissipate entropy.
\end{corollary}
{The above corollary states in particular that $u$ is the unique Kruzkov entropy solution to the Cauchy problem (when its distributional source term is assigned). We mention that in the case of the quadratic flux this statement can be derived by~\cite{Pinamonti}: the authors provide a smooth approximation for which also the source term is converging in $\Ll^{1}(\R^{2})$, refining a construction in~\cite{MV} which extends to the Heisenberg group a technique originally introduced for the Euclidean setting by~\cite{DGCP}.
The construction we adopt here is more direct but rougher: sources do not converge.
}

\begin{proof}[Proof of Corollary~\ref{C:lagrnodiss}]
We exploit the approximation $\{u_{k}(t,x)\}_{k\in\N}$ given in Lemma~\ref{L:monotoneApproxiamtion}.
Consider any entropy-entropy flux pair $\eta,q\in C^{1}(\R)$, that is $q'(z)=\eta'(z)f'(z)$.
As each $u_{k}$ is a function of bounded variation, by Vol'pert chain rule 
\[
\Dt\eta(u_{k}) + \Dx (q(u_{k})) 
=
\eta'(u_{k})\left(\Dt u_{k} + \Dx f(u_{k})\right) 
=
\eta'(u_{k})\gd_{k} \ ,
\]
where we set $\gd_{k}=\Dt u_{k} + \Dx f(u_{k})$. Owing to Lemma~\ref{L:BVcontsol2}, each $\gd_{k}$ is given by a function which is bounded by the constant $G$ in the assumption of the present corollary.
Since $u_{k}$ converges uniformly to $u$, one has
\[
\eta'(u_{k})\gd_{k}=\pt \eta(u_{k}) + \px (q(u_{k}))\quad \to\quad\pt\eta(u) + \px (q(u))\quad
\text{in $\mathcal D'(\Omega)$.}
\]
Being the sequence $\{\gd_{k}\}_{k\in\N}$ uniformly bounded, Banach-Alaoglu theorem implies that there there exists a subsequence $w^{*}$-converging to some function $\gd\in\distrBdd(\Omega)$, $\norm{\gd}_{\distrBdd(\Omega)}\leq G$: then necessarily
\[
\pt\eta(u) + \px (q(u)) 
=
\eta'(u)\gd .
\]

The Eulerian source $\gd$ and the Lagrangian source $\gx$ can be identified also in the limit under uniform convexity assumptions on the f{}lux, see \cite[\S~2.2]{file2ABC}. Under the negligibility assumption on the inf{}lection points of $f$~\eqref{hyp:H}, they are just compatible: see \cite[\S~3 and \S~4.2]{file2ABC} .
\end{proof}

\begin{proof}[Proof of Lemma~\ref{L:monotoneApproxiamtion}]
We construct an approximation of $u$ by a patching procedure. One needs first to construct the approximation on a patch, which is a strip delimited by two characteristics. In it, we require that at each fixed time the approximating function is monotone in $x$, and it coincides with $u$ on the boundary of the strip.
This allows to work with continuous functions having bounded variation.
Repeating the construction in adjacent strips, when they get thinner the approximating functions converge to $u$ uniformly.

We recall that $u$ can be assumed compactly supported, see Notation~\ref{N:notationcompactsupp}.

We expose first the limiting procedure for constructing a monotone approximation within each strip, and then a second limiting procedure for converging to $u$ when strips become thinner.
In the first step we describe the second limiting procedure, which is simpler, while from the second step on we describe how to provide the monotone approximations.

\firststep
\step{Patches decomposition}
Fix two characteristics $\chi_{y_{1}}^{\reu}(t)$, $\chi_{y_{2}}^{\reu}(t)$ and define the strip
\begin{equation}
\label{E:strip}
S_{y_{1}y_{2}}
=
\left\{
(t,x)\in\R^{+}\times\R\ : \ 
\chi(t,y_{1})\leq x \leq\chi(t,y_{2})
\right\} 
\qquad\text{where $y_{1}\leq y_{2}$}.
\end{equation}
If one choses for example $y_{i}=i\delta$ for $i\in\Z$ and some $\delta>0$, then one has the decomposition
\[
\R^{+}\times\R = \cup_{i\in\Z} S_{y_{i}y_{i+1}}
.
\]
Let $|y_{i+1}-y_{i}|\leq \delta$.
We construct in the next steps continuous functions $u^{\delta}$ which are
\begin{enumerate}
\item \label{item:lagrangian}Lagrangian solutions, with a new Lagrangian source still bounded by $G$;
\item \label{item:boundary}equal to $u$ on the curves $\chi_{y_{i}}^{\reu}(t)$, $i\in\Z$;
\item \label{item:mon1}nondecresing in the $x$-variable in each open $t$-section of the set \[\{(t,x):\ \chi_{y_{i}}^{\reu}(t)<x<\chi_{y_{i+1}}^{\reu}(t),\ u(i_{\chi(y_{i})}(t))\leq u(i_{\chi(y_{i+1})}(t))\};\]
\item \label{item:mon2}nonincreasing in the $x$-variable in each open $t$-section of the set \[\{(t,x):\ \chi_{y_{i}}^{\reu}(t)<x<\chi_{y_{i+1}}^{\reu}(t),\ u(i_{\chi(y_{i})}(t))\geq u(i_{\chi(y_{i+1})}(t))\};\]
\item \label{item:distance}$\norm{u^{\delta}-u}_{\distrBdd}\leq \omega(\delta)$, where $\omega(\delta)$ is a $\delta$-modulus of uniform continuity of \[U(t,y)=u(i_{\chi_{y}^{\reu}}(t)).\]
\end{enumerate}
From the monotonicity properties~\eqref{item:mon1}-\eqref{item:mon2}, if we apply the slicing theory for $\BV$-functions we notice that the functions $u^{\delta}$ have locally bounded variation.
Owing to Lemma~\ref{L:BVcontsol2} they are also Eulerian solutions with source terms which are uniformly bounded by $G$, the uniform bound for the Lagrangian sources owing to~\eqref{item:lagrangian}. By the uniform estimates~\eqref{item:distance} and the constrain~\eqref{item:boundary} on the boundary of the strips, they converge uniformly to $u$ as $\delta\downarrow0$, proving the thesis.

\step{Monotone modification within a patch}
We start the iterative procedure for constructing the approximations $u^{\delta}$ having bounded variation, that we describe within a fixed strip 
\begin{equation}
S_{y_{1}y_{2}}
=
\left\{
(t,x)\ : \ 
\chi(t,y_{1})< x< \chi(t,y_{2})
\right\} .
\end{equation}
We modify in $S_{y_{1}y_{2}}$ the given Lagrangian continuous solution in order to get a new continuous function which is still a Lagrangian solution, for a different source which is still bounded by $G$.
The additional property that we are trying to get, piecewise, is a monotonicity in the $x$ variable when $t$ is fixed: we fix the values on the boundary curves $\chi_{y_{1}}^{\reu}(t)$, $\chi_{y_{2}}^{\reu}(t)$; inside the stripe $S_{y_{1}y_{2}}$ we aim at substituting at each time $u^{\reu}_{t}(x)$ with a function $\tilde u^{\reu}_{t}(x)$ i) which is monotone in the $x$ variable with values from $u(i_{\chi_{(y_{1})}}(t))$ to $ u(i_{\chi(y_{2})}(t))$ and ii) which is still a Lagrangian solution with source term bounded by $G$.
This new function $\tilde u$ is now defined inside the stripe $S_{y_{1}y_{2}}$ with a limiting procedure pictured in Figure~\ref{fig:tagli} below.
\begin{figure}[ht]
\centering
\includegraphics[width=.43\linewidth, height=.2\linewidth]{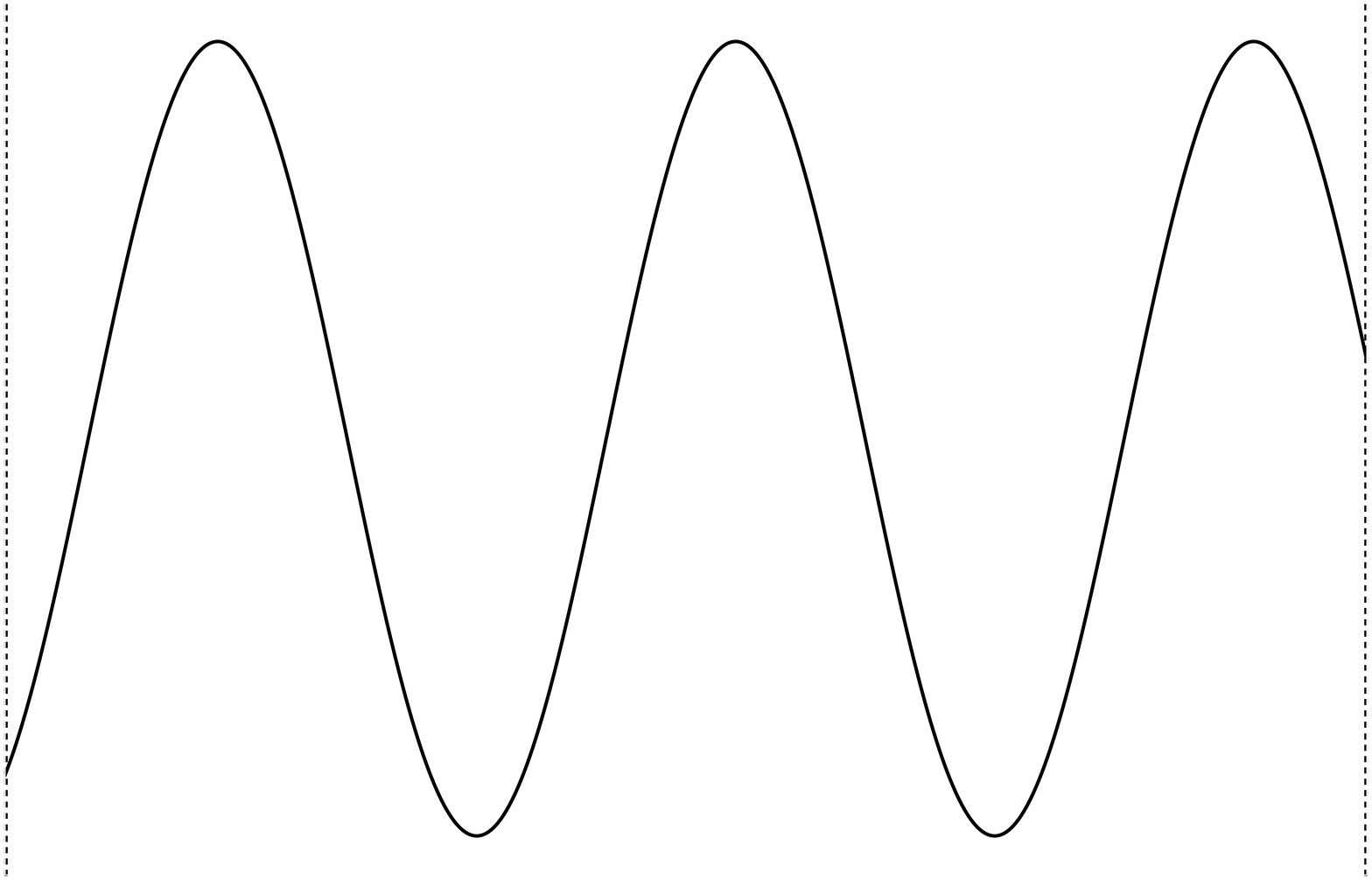}
\hspace{.1\linewidth}
\includegraphics[width=.43\linewidth,height=.2\linewidth]{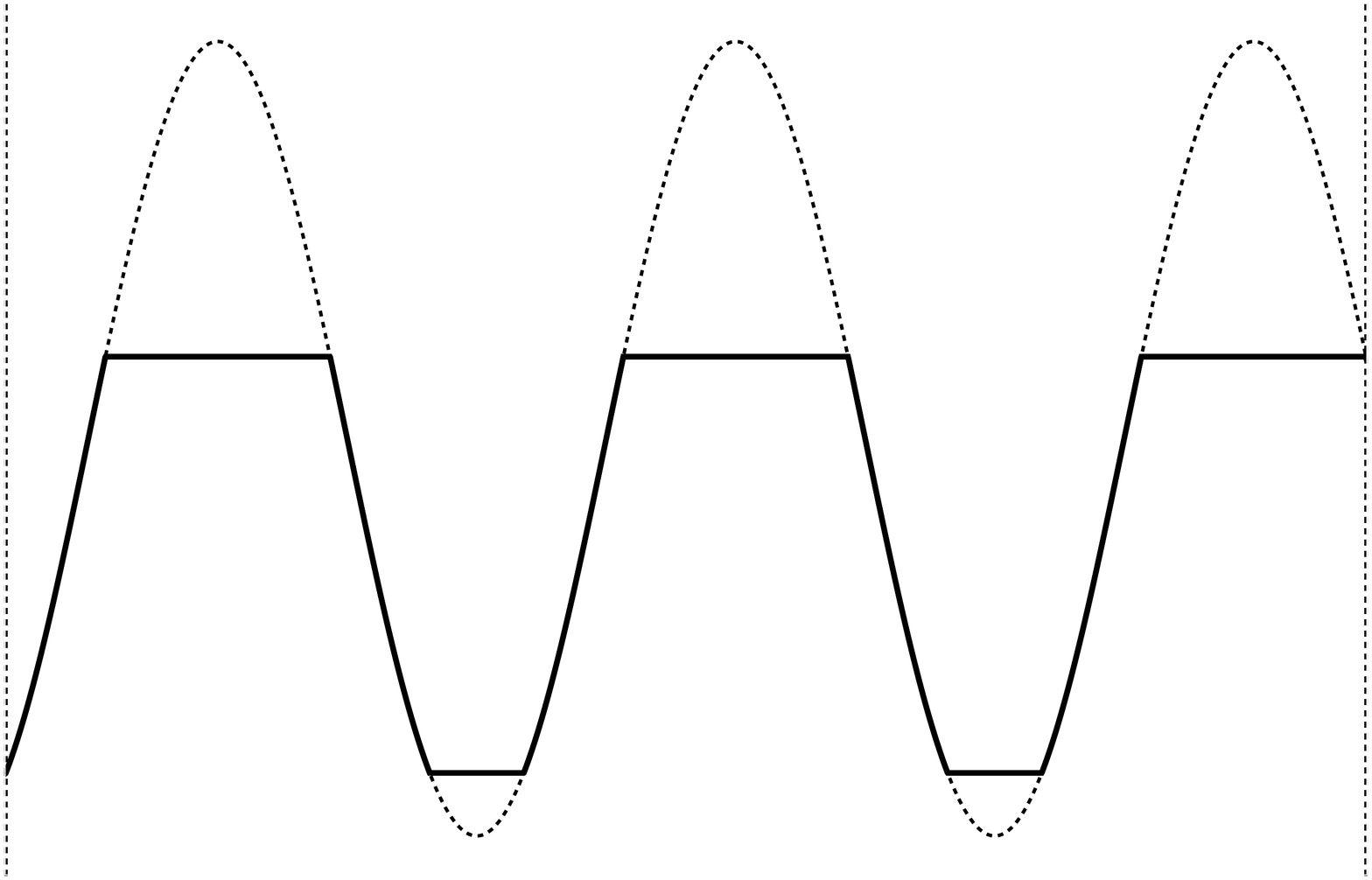}\\
\includegraphics[width=.43\linewidth,height=.2\linewidth]{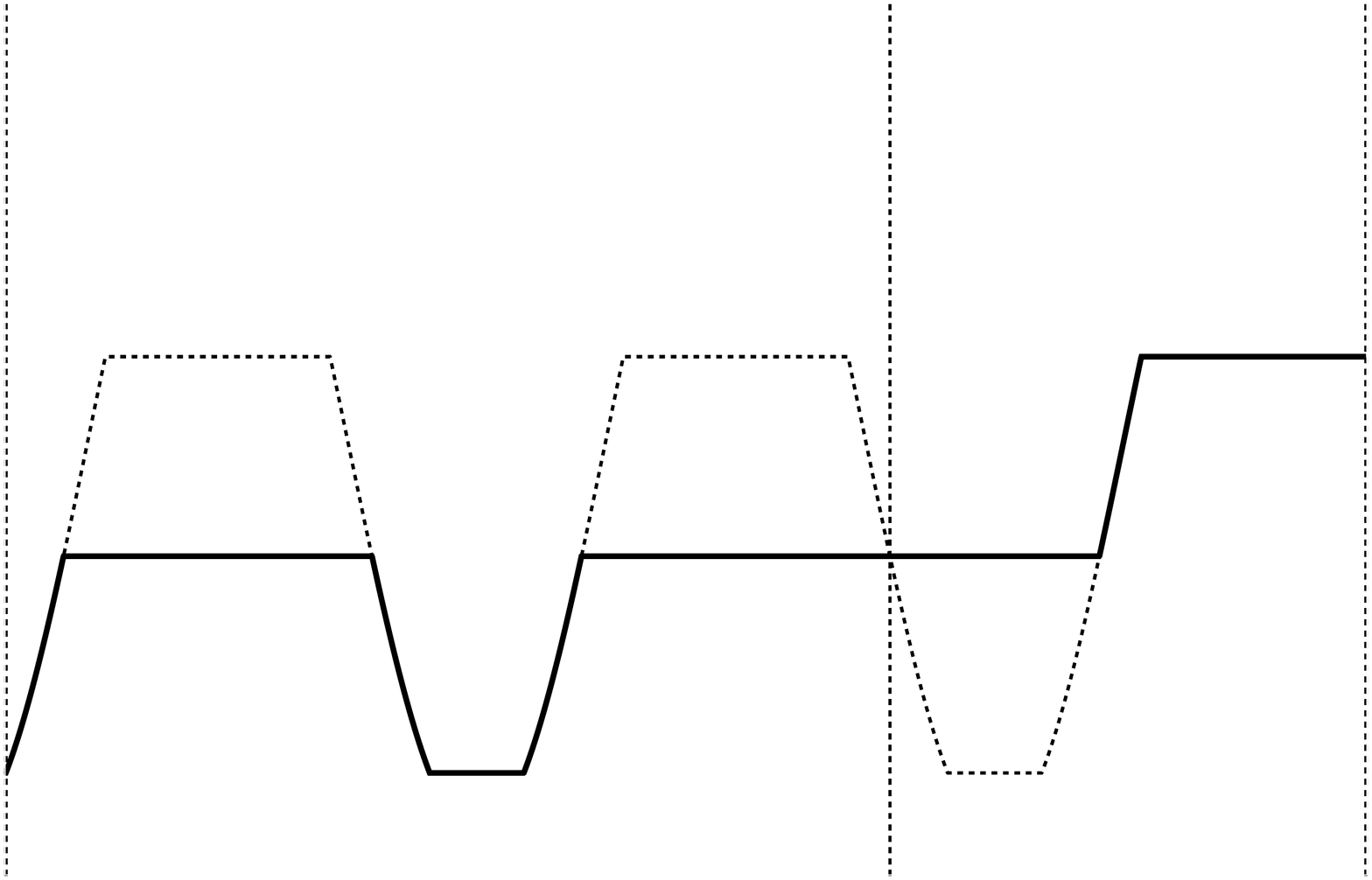}
\hspace{.1\linewidth}
\includegraphics[width=.43\linewidth,height=.2\linewidth]{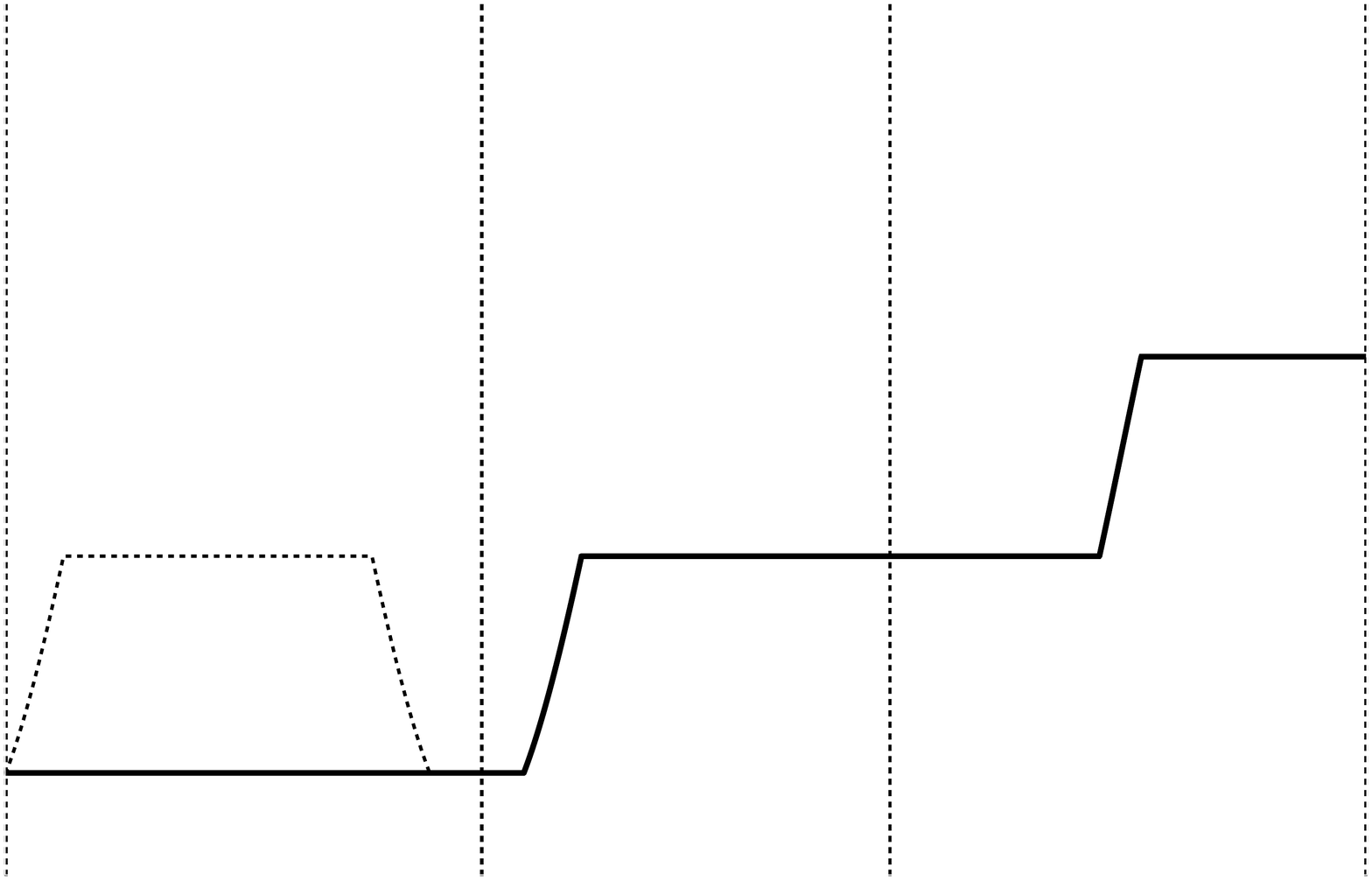}\
\caption{First steps of the iterative procedure: the bold line is the approximating function obtained by successive cuttings reaching monotonicity properties}
\label{fig:tagli}
\end{figure}

Let $\{(t_{j}, x_{j})\}_{j\in\N}$ be a dense sequence of points within the strip $\{S_{y_{1}y_{2}}\}$.
The function $\tilde u$ is defined within the strip $\{S_{y_{1}y_{2}}\}$ as a uniform limit of functions $\tilde u_{j}$, for $j\in\N$, which we assign now recursively.
We first state the basic operation that we will perform.

\begin{claim}[Basic cut]
\label{C:basiccut}
Suppose $u$ is a Lagrangian solution of~\eqref{E:basicPDE} with source term bounded by $G$ and fix a characteristic curve $\bar \gamma=\chi(t,\bar y)$, where $\chi$ is a Lagrangian parameterization of $u$. Then the two truncated functions, respectively from above and from below,
\[
 u_{M}(t,x):= u(t,x)\wedge u(i_{\bar\gamma}(t)),
\qquad  u_{m}(t,x):=u(i_{\bar\gamma}(t))\vee u(t,x)
\] 
are still Lagrangian solutions of~\eqref{E:basicPDE} with source term bounded by $G$.
\end{claim}

We postpone the proof of Claim~\ref{C:basiccut} to Page~\pageref{proof:basiccut}. The basic cut allows the iterative procedure:
\begin{itemize}
\item Set $m(t):=\min\{u(i_{\chi_{(y_{1})}}(t)), u(i_{\chi_{(y_{2})}}(t))\}$ and $M(t):=\max\{u(i_{\chi_{(y_{1})}}(t)), u(i_{\chi_{(y_{2})}}(t))\}$.
\item Set $\tilde u_{0}(t,x):= u(t,x)$ for $(t,x)\in S_{y_{1}y_{2}}$ and fix $\bar\gamma_{0}(t):=\chi(t,y_{1}(t))$. 
\item Let $j\in\N$. Set $v_{j}(t)=u(i_{\bar \gamma_{j-1}}(t))$ and define the truncated function
\[ 
\tilde u_{j}(t,x)
:=
\begin{cases}
v_{j}(t)\vee \tilde u_{j-1}(t,x) \wedge M(t) &\text{if $u(i_{\chi_{(y_{1})}}(t))=M(t)$ and $\chi(t,y_{1})\leq x\leq \bar\gamma_{j-1}(t)$} ,\\
m(t)\vee \tilde u_{j-1}(t,x)\wedge v_{j}(t)  &\text{if $u(i_{\chi_{(y_{1})}}(t))=m(t)$ and $\chi(t,y_{1})\leq x\leq \bar\gamma_{j-1}(t)$} ,\\
 v_{j}(t)\vee \tilde u_{j-1}(t,x) \wedge M(t) &\text{if $u(i_{\chi_{(y_{2})}}(t))=M(t)$ and $\bar\gamma_{j-1}(t)\leq x\leq \chi(t,y_{2})$} ,\\
m(t)\vee \tilde u_{j-1}(t,x)\wedge v_{j}(t)  &\text{if $u(i_{\chi_{(y_{1})}}(t))=m(t)$ and $\bar\gamma_{j-1}(t)\leq x\leq \chi(t,y_{2})$} .
\end{cases} 
\] 
Basically, the strip $S_{y_{1}y_{2}}$ is divided by $\bar\gamma_{j-1}$ into two sub-strips and in each sub-strip $\tilde u_{j-1}$ is truncated from above and from below by the two boundary values on the sub-strips. 
Owing to Claim~\ref{C:basiccut} the function $\tilde u_{j}$ is a Lagrangian solution of~\eqref{E:basicPDE} with source term bounded by $G$. 
By the definition of Lagrangian solution one can fix $\bar\gamma_{j}$ as
\begin{itemize}
\item a characteristic curve of $\tilde u_{j}$ 
\item through the point $(t_{j},x_{j})$ 
\item along which $\tilde u_{j}$ is $G$-Lipschitz continuous and
\item which does not cross the previously chosen characteristics ${\chi{(\cdot,y_{1})}}$, ${\bar\gamma_{1}(t)}$,\dots, ${\bar\gamma_{j-1}(t)}$, ${\chi{(\cdot,y_{2})}}$; this means that any two characteristics of this set lie always on the same side of the pane with respect to each other, when they differ.
\end{itemize}
Note before proceeding that $\tilde u_{j}$ is monotone at each fixed time on the $j+2$ points
\begin{equation}
\label{E:pointmonotone}
i_{\chi_{(y_{1})}}(t),\quad \left\{i_{\bar\gamma_{i}(t)}\right\}_{i=1,\dots,j},\quad i_{\chi_{(y_{2})}}(t).
\end{equation}
Because of this monotonicity, at later steps of the iteration the values of the function $\tilde u_{j}$ on the points~\eqref{E:pointmonotone} are not changed.
\end{itemize}
We obtained with the iterative procedure that
\begin{enumerate}
\item each continuous function $\tilde u_{j}$ is a Lagrangian solution and $G$ still bounds its source, thanks to Claim~\ref{C:basiccut};
\item the (whole) sequence $\{\tilde u_{j}\}_{j\in\N}$ converges uniformly on $S_{y_{1}y_{2}}$ to a function $\tilde u$, because $u$ is uniformly continuous and the cutting procedure preserves the modulus of continuity: similarly to the next item, for $h>j$ one has the estimate
\begin{equation*}
\begin{split}
\norm{\tilde u_{j} -\tilde u_{h}}_{\distrBdd(S_{y_{1}y_{2}})}
&\leq  \omega\left(\diam\left(S_{y_{1}y_{2}}\setminus\operatorname{Im}\left\{i_{\chi_{(y_{1})}}, \left\{i_{\bar\gamma_{i}}\right\}_{i=1,\dots,j},\quad i_{\chi_{(y_{2})}} \right\}\right)\right)
\end{split}
\end{equation*}
and the curves~\eqref{E:pointmonotone} become dense in $S_{y_{1}y_{2}}$ by construction, so the RHS goes to $0$.
\item $\norm{u -\tilde u}_{\distrBdd(S_{y_{1}y_{2}})}\leq \omega(y_{2}-y_{1})$, where $\omega$ is a modulus of uniform continuity of $U(t,y)$, because by construction we have
\begin{equation}
\label{E:unifomrmestimateutilde}
\begin{split}
\norm{u -\tilde u}_{\distrBdd(S_{y_{1}y_{2}})}
&\leq
\sup_{(t,x)\in S_{y_{1}y_{2}}} \left\{ | u(t,x) - u(i_{\chi(y_{1})}(t)) |,| u(t,x) - u(i_{\chi(y_{2})}(t)) |\right\}
\\
&\leq 
\omega(y_{2}-y_{1}) \leq \omega(\delta);
\end{split}
\end{equation}
\item $\tilde u$ is still a Lagrangian solution and that $G$ still bounds its source by Corollary~\ref{C:unifConverengeLagr};
\item $\tilde u$ is monotone in the $x$-variable at each $t$ fixed because each $\tilde u_{j}$ is monotone on the points~\eqref{E:pointmonotone}, which become dense in the interval $[\chi(t,y_{1}),\chi(t,y_{1})]$.\qedhere
\end{enumerate}
\end{proof}
We are finally left with the proof of Claim~\ref{C:basiccut}.
\begin{proof}[Proof of Claim~\ref{C:basiccut}]
\label{proof:basiccut}
Owing to Lemma~\ref{L:lipcharisLagr}, $ u_{m}(t,x):=u(i_{\bar\gamma}(t))\vee u(t,x)$ is a Lagrangian solution provided that we exhibit a characteristic curve, along which $u_{m}$ is $G$-Lipschitz continuous, through any point $(t,x)$.
For simplifying the exposition suppose that the set $\{f'(u_{m})<f'(u)\}$  is empty, if not this region is treated similarly to below as a second step.
Set
\[
\lambda_{m}(t,x):=f'(u_{m}(t,x))\quad\geq\quad \lambda(t,x):=f'(u(t,x)).
\]

Consider the set of $C^{1}$-curves through a point $(\bar t,\bar x)$ defined by
\begin{equation*}
\Gamma
=
\clos\left(\left\{\begin{aligned}\gamma\in C^{1}(\R^{+})\quad \Big|\quad \gamma(\bar t)=\bar x, \ \exists j,\ \exists I_{1},\dots, I_{j},\phantom{OOOOOOOOO} \\
\exists J_{0},\dots, J_{j} \ : \  J_{0}\cup I_{1}\cup J_{1}\cup \dots J_{j}=\R^{+} \\
 \text{ and } \exists c_{1},\dots, c_{j},z_{0}, \dots,z_{j}\text{ such that}\quad\\ \begin{cases}\gamma(t)=\chi(t,\bar\gamma(t))+c_{i} ,\ \lambda_{m}(i_{\gamma}(t))>\lambda(i_{\gamma}(t)) &\text{if $t\in I_{i}$, $i=1,\dots,j$}\\ \gamma(t)=\chi(t,z_{i})&\text{if $t\in J_{i}$, $i=0,\dots,j$}\end{cases}\end{aligned}\right\}\right)^{C^{1}}
\end{equation*}
for suitable intervals $ I_{1},\dots, I_{j}, J_{0},\dots, J_{j}$ and values $c_{1},\dots, c_{j},z_{0}, \dots,z_{j}$ depending on the $\gamma$.
The set of curves is not empty, for example because the curves of the Lagrangian parameterization $\chi$ through $(t,x)$ belongs to it.
Moreover, the function $u_{m}$ is $G$-Lipschitz continuous on each $\gamma$ described within the brackets: indeed, where $\lambda_{m}>\lambda$ necessarily $u_{m}\neq u$ and hence by definition of $u_{m}$
\[
u_{m}(\gamma(t))
=
\begin{cases}
u(i_{\bar\gamma}(t)) &\text{if $t\in I_{1}\cup\dots\cup I_{j}$}\\ 
u(i_{\chi(z_{i})}(t))&\text{if $t\in J_{i}$, }i=0,\dots,j
\end{cases}
\]
and each function $u(i_{\bar\gamma}(t)) ,u(i_{\chi(z_{0})}(t)), \dots, u(i_{\chi(z_{j})}(t)) $ is $G$-Lipschitz continuous by assumption.
As a consequence, $u_{m}$ is $G$-Lipschitz continuous on each element of the closure $\Gamma$. The curves
\[
\gamma_{m}^{+}(t):=\max_{\gamma\in\Gamma}\gamma(t) \text{ for $t\geq \bar t$}, \qquad\gamma_{m}^{-}(t):=\min_{\gamma\in\Gamma}\gamma(t) \text{ for $t< \bar t$}
\] 
still belongs to $\Gamma$, if suitably prolonged. In particular, $u_{m}$ is $G$-Lipschitz continuous along $\gamma_{m}$.
This concludes the proof observing that $\gamma_{m}^{+}$ is necessarily a forward characteristic curve of $u_{m}$, and $\gamma_{m}^{-}$ a backward one, through $(\bar t,\bar x)$. Indeed, each $\gamma$ in the definition of $\Gamma$ is a $C^{1}$ curve whose slope satisfies
\[
\dot\gamma\in\{\lambda\circ i_{\gamma}, \lambda_{m}\circ i_{\gamma}\} \qquad\Rightarrow\qquad \dot\gamma^{+}_{m},\dot\gamma^{-}_{m}\in\{\lambda\circ i_{\gamma}, \lambda_{m}\circ i_{\gamma}\}.
\]
If by absurd we had $\dot\gamma_{m}^{+}(t')=\lambda\circ i_{\gamma_{m}^{+}}(t')< \lambda_{m}\circ i_{\gamma_{m}^{+}}(t')$ at some time $t'>\bar t$ then we would contradict the extremity in the definition of $\gamma_{m}^{+}$: considering $s'=\bar t \vee\max\{s<t'\ :\ \dot\gamma^{+}_{m}(s)=\lambda_{m}\circ i_{\gamma^{+}_{m}}(s))\} $, one can verify that there is an element of $\Gamma$ which satisfies
\[
\tilde\gamma_{m}^{+}(t)
=
\begin{cases}
\gamma_{m}^{+}(t) &\text{for $\bar t\leq t\leq s'$}\\
 \bar\gamma_{\tau}(t):=\bar\gamma(t )-[ \bar\gamma(s' )-\gamma^{+}_{m}(s')] &\text{for $s'<t\leq  s'' $,}
\end{cases}
\]
where $s'':=t' \wedge\min\{s>s'\ :\ \lambda_{m}\circ i_{ \bar\gamma_{\tau}}(s))\neq\lambda\circ i_{ \bar\gamma_{\tau}}(s))\}$.
Notice that 
$\tilde\gamma_{m}^{+}$ is bigger than $\gamma_{m}$ at time $s''$ because $\dot{\tilde\gamma}_{m}^{+}(s')=\gamma_{m}^{+}(s')$ and by construction  
\[\dot{\tilde\gamma}_{m}^{+}=\lambda_{m}\circ i_{\bar\gamma}=\lambda_{m}\circ i_{\gamma^{+}_{m}}(s)>\lambda\circ i_{\gamma_{m}^{+}}=\dot\gamma_{m}^{+}\text{ for }s'<t<s''.\]
The other possibilities contradict analogously the definition of $\gamma_{m}$.
\end{proof}

\section{Distributional solutions are broad solutions, if inf{}lections are negligible}
\label{S:distributionaltoBroad}

We provide in this section regularity results holding under the assumption that $f$ has negligible inf{}lection points: we prove that $u$ is Lipschitz continuous along every characteristic curve and there exists a universal source which is fine for every Lagrangian parameterization one chooses.

Without assumptions on inf{}lection points, later \S~\ref{S:distrareLagrangian} shows that distributional solutions of
\begin{equation*}
\tag{\ref{E:basicPDE}}
\pt u(t,x) + \px (f(u(t,x))) =  \gd(t,x)
\qquad f\in C^{2}(\R),
\qquad 
|\gd(t,x)|\leq G
\end{equation*}
are also Lagrangian solutions.
Being a Lagrangian solution allows to study $u$ with tools from ODEs, but it is not completely satisfactory by itself because one should be a priory careful in choosing the right Lagrangian parameterization, and the correct source related to the parameterization: the results of the present section are richer because here \emph{any} Lagrangian parameterization is allowed.

Being local arguments, we simplify the setting posing $\Omega=\R^+\times\R$, $u$ compactly supported.

\subsection{Lipschitz regularity along characteristics}
\label{Ss:Lipregalongchar}

In the present section we point out that $u$ is Lipschitz continuous along characteristic curves if inf{}lection points of $f$ are negligible:
\begin{equation}
\label{E:nonvanishingCondition}
\Ll^{1}(\clos({\infl(f)}))=0.
\tag{H}
\end{equation}
See Example~\cite[\S~4.2]{file2ABC} for a counterexample when~\eqref{E:nonvanishingCondition} fails.

\begin{theorem}
\label{T:sharpLipschitzreg}
Assume that the non-vanishing condition~\eqref{E:nonvanishingCondition} holds.
Then any continuous distributional solution $u$ of~\eqref{E:basicPDE} is $G$-Lipschitz continuous along any characteristic curve of $u$.
\end{theorem}


\begin{proof}
It takes a while to realize that the following is a partition of the real line into the regions
\begin{align*}
&D^{+}:=\ri\left(\big\{z\ \big|\ \exists \bar h>0:\ \forall h\in[-\bar h, \bar h]\quad f(z+h)-f(z)\geq f'(z)h\big\}\right),
\\
&D^{-}:=\ri\left(\big\{z\ \big|\ \exists \bar h>0:\ \forall h\in[-\bar h, \bar h]\quad f(z+h)-f(z)\leq f'(z)h\big\}\right),
\\
&N:=\R\setminus \left(D^{+}\bigcup D^{-}\right) \equiv \clos(\infl(f)).
\end{align*}
By assumption $N$ is Lebesgue negligible.

Consider any characteristic curve $i_{\gamma}(t)=(t,\gamma(t))$, where $\dot \gamma(t)=f'(u(t,\gamma(t)))$, $t\in\R^{+}$.
We first follow a similar computation in~\cite{Daf} which shows that $u(t,\gamma(t))$ is $\norm{g}_{\infty}$-Lipschitz continuous on the connected components of the open set $(u\circ i_{\gamma})^{-1}(D^{+})$, as in~\cite{BCSC}. 

\begin{figure}
\label{F:aree}
\centering
\begin{tikzpicture}[thick]
\def\assexLung{5.5}
\def\asseyLung{5}
\def\xa{2.5}
\def\xb{4.5}
\def\ta{2}
\def\tb{4}
\def\tc{4.5}
\def\tacca{2pt}
\def\xshiftvertice{\xa/2+\xb/2}
\draw[-stealth](-0.3,0)--(\assexLung+0.3,0) node[below]{$x$};
\draw[-stealth](0,-0.3)--(0,\asseyLung+0.3) node[left]{$t$};
\draw(0+\tacca,\ta)--(0-\tacca,\ta) node[left]{$\sigma$};
\draw(0+\tacca,\tb)--(0-\tacca,\tb) node[left]{$\tau$};

\fill[xshift=0.4cm, yshift=\ta cm, color =gialloLimone] (\xa,0) .. controls (\xa,1)  and (\xa+0.4 ,\ta-1) .. (\xa+0.4,\ta)--(\xb,0) .. controls (\xb,1)  and (\xb+0.4 ,\ta-1) .. (\xb+0.4,\ta)--cycle;
\fill[xshift=0.4cm, yshift=\ta cm, color =gialloLimone] (\xa,0) .. controls (\xa,1)  and (\xa+0.4 ,\ta-1) .. (\xa+0.4,\ta)--(\xb+0.4,\ta) .. controls (\xb+0.4,\ta)  and (\xb,1) .. (\xb,0)--cycle;
\draw (\xa,0) .. controls (\xa,1)  and (\xa+0.4 ,\ta-1) .. (\xa+0.4,\ta);
\draw[xshift=0.4cm, yshift=\ta cm] (\xa,0) .. controls (\xa,1)  and (\xa+0.4 ,\ta-1) .. (\xa+0.4,\ta);
\draw[xshift=0.8cm, yshift=2*\ta cm] (\xa,0) .. controls (\xa,0.5)  and (\xa+0.2 ,\ta-0.5-1.1) .. (\xa+0.2,\ta-1.1);
\begin{scope}[xshift=\xb cm-\xa cm, dashed]
	\draw (\xa,0) .. controls (\xa,1)  and (\xa+0.4 ,\ta-1) .. (\xa+0.4,\ta);
	\draw[xshift=0.4cm, yshift=\ta cm] (\xa,0) .. controls (\xa,1)  and (\xa+0.4,\ta-1) .. (\xa+0.4,\ta);
	\draw[xshift=0.8cm, yshift=2*\ta cm] (\xa,0) .. controls (\xa,0.5)  and (\xa+0.2 ,\ta-0.5-1.1) .. (\xa+0.2,\ta-1.1);
\end{scope}
\draw[thin, dashed] (0,\ta)--(\xb+0.4,\ta);
\draw[thin, dashed] (0,\tb)--(\xb+2*0.4,\tb);
\node[left=-0.2cm] at (\xa,1) {$\gamma(t)$};
\node[left=-0.2cm] at (\xb,1) {$\gamma(t)+\varepsilon$};
\node[above=.25cm, right=0.4cm] at (\xa,\ta) {$(\bar t,\bar x)$};
\end{tikzpicture}
\begin{tikzpicture}[thick]
\def\assexLung{5.5}
\def\asseyLung{5}
\def\xa{2.5}
\def\xb{0.5}
\def\ta{2}
\def\tb{4}
\def\tc{4.5}
\def\tacca{2pt}
\def\xshiftvertice{\xa/2+\xb/2}
\draw[-stealth](-0.3,0)--(\assexLung+0.3,0) node[below]{$x$};
\draw[-stealth](0,-0.3)--(0,\asseyLung+0.3) node[left]{$t$};
\draw(0+\tacca,\ta)--(0-\tacca,\ta) node[left]{$\sigma$};
\draw(0+\tacca,\tb)--(0-\tacca,\tb) node[left]{$\tau$};
\fill[xshift=0.4cm, yshift=\ta cm, color=gialloLimone] (\xa,0) .. controls (\xa,1)  and (\xa+0.4 ,\ta-1) .. (\xa+0.4,\ta)--(\xb,0) .. controls (\xb,1)  and (\xb+0.4 ,\ta-1) .. (\xb+0.4,\ta)--cycle;
\fill[xshift=0.4cm, yshift=\ta cm, color=gialloLimone] (\xa,0) .. controls (\xa,1)  and (\xa+0.4 ,\ta-1) .. (\xa+0.4,\ta)--(\xb+0.4,\ta) .. controls (\xb+0.4 ,\ta-1)  and (\xb,1) .. (\xb,0)--cycle;
\draw (\xa,0) .. controls (\xa,1)  and (\xa+0.4 ,\ta-1) .. (\xa+0.4,\ta);
\draw[xshift=0.4cm, yshift=\ta cm] (\xa,0) .. controls (\xa,1)  and (\xa+0.4 ,\ta-1) .. (\xa+0.4,\ta);
\draw[xshift=0.8cm, yshift=2*\ta cm] (\xa,0) .. controls (\xa,0.5)  and (\xa+0.2 ,\ta-0.5-1.1) .. (\xa+0.2,\ta-1.1);
\begin{scope}[xshift=\xb cm-\xa cm, dashed]
	\draw (\xa,0) .. controls (\xa,1)  and (\xa+0.4 ,\ta-1) .. (\xa+0.4,\ta);
	\draw[xshift=0.4cm, yshift=\ta cm] (\xa,0) .. controls (\xa,1)  and (\xa+0.4,\ta-1) .. (\xa+0.4,\ta);
	\draw[xshift=0.8cm, yshift=2*\ta cm] (\xa,0) .. controls (\xa,0.5)  and (\xa+0.2 ,\ta-0.5-1.1) .. (\xa+0.2,\ta-1.1);
	\end{scope}
\draw[thin, dashed] (0,\ta)--(\xa+0.4,\ta);
\draw[thin, dashed] (0,\tb)--(\xa+2*0.4,\tb);
\node[right=0.2cm] at (\xa,1) {$\gamma(t)$};
\node[right=0.2cm] at (\xb,1) {$\gamma(t)-\varepsilon$};
\node[above=.25cm, right=0.4cm] at (\xa,\ta) {$(\bar t,\bar x)$};
\end{tikzpicture}
\caption{Balances on characteristic regions}
\end{figure}
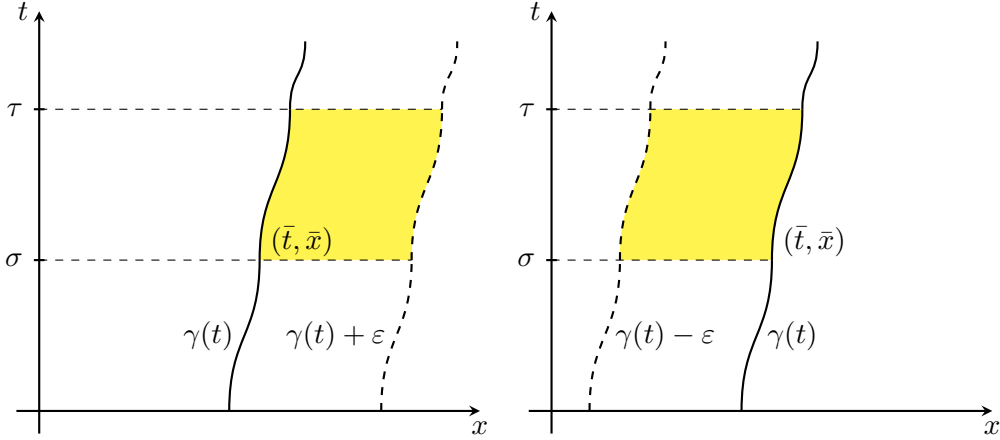
Focus on the domain bounded by the curves $i_{\gamma}(t)$, $i_{\gamma+\varepsilon}(t)$ between times $\sigma<\tau$.
The equality
\begin{subequations}
\label{E:Dafermoscomputation}
\begin{equation}
\label{E:Dafermoscomputation1}
\begin{split}
\int_{\gamma(\tau)}^{\gamma(\tau)+\varepsilon} u(\tau,x) dx
-\int_{\gamma(\sigma)}^{\gamma(\sigma)+\varepsilon} u( \sigma,x)dx
-\int_{\sigma}^{\tau}\int_{\gamma(t)}^{\gamma(t)+\varepsilon} g( t,x)dxdt
\quad\\
=
-\int_{\sigma}^{\tau}\left\{
	\left[f(u(i_{\gamma+\varepsilon}(t))) - f(u(i_{\gamma}(t))\right]
	-f'(u(i_{\gamma}(t))) 
		\left[u(i_{\gamma+\varepsilon}(t))-u(i_{\gamma}(t))\right]
\right\}dt
\\
=
-\int_{\sigma}^{\tau}\left\{
	f''(\xi) 
		\frac{\left[u(i_{\gamma+\varepsilon}(t))-u(i_{\gamma}(t))\right]^{2}}{2}
\right\}dt
		\qquad \xi(t)\in [u(i_{\gamma}(t))-u(i_{\gamma+\varepsilon}(t))]
\end{split}
\end{equation}
can be obtained integrating suitable test functions converging to the indicator of the region (Figure~\ref{F:aree}).
If either $u(\sigma,\gamma(\sigma))$ or $u(\tau,\gamma(\tau))$ belong to $D^{+}$, by definition of $D^{+}$ for $\tau-\sigma$ sufficiently small the RHS is nonpositive: we obtain thus the inequality
\begin{equation}
\label{E:Dafermoscomputation2}
\int_{\gamma(\tau)}^{\gamma(\tau)+\varepsilon} u(\tau,x)dx
-\int_{\gamma(\sigma)}^{\gamma(\sigma)+\varepsilon} u( \sigma,x)dx
\leq
\int_{\sigma}^{\tau}\int_{\gamma(t)}^{\gamma(t)+\varepsilon} g( t,x)dxdt
\leq
\varepsilon\norm{g}_{\infty}  |\tau-\sigma| .
\end{equation}
\end{subequations}
Dividing by $\varepsilon$, by the continuity assumptions on $u$ in the limit as $\varepsilon\downarrow0$ this yields
\[
u(\tau,\gamma(\tau))-u(\sigma,\gamma(\sigma))
\leq
\norm{g}_{\infty}  |\tau-\sigma| .
\]
The converse inequality is obtained by considering the similar region between $i_{\gamma-\varepsilon}(t)$, $i_{\gamma}(t)$: indeed this lead to an equation analogous to~\eqref{E:Dafermoscomputation}, but with RHS having opposite sign.

We conclude from the above analysis that $u(i_{\gamma}(t))$ is $\norm{g}_{\infty} $-Lipschitz continuous in a neighborhood of any point belonging to the inverseimage of $D^{+}$.
The same holds in an analogous way for $D^{-}$.
This local Lipschitz continuity can be equivalently stated by the inequality
\begin{equation}
\label{E:loclipg}
\Ll^{1} \left( u( i_{\gamma}^{}(B)) \right) \leq G\Ll^{1} \left(B \right)
\end{equation}
for all Borel subsets $B$ of the open set $O=\left(u\circ i_{\gamma}\right)^{-1}(D^{+}\cup D^{-})$.

The thesis finally follows by the negligibility of $N$: for every $t_{1}\leq t_{2}$
\begin{align*}
|u( i_{\gamma}^{}(t_{2}))-u( i_{\gamma}^{}(t_{1}))|
&\leq
\Ll^{1} \left( u( i_{\gamma}^{}([t_{1},t_{2}])) \right)
\\
&=
\Ll^{1} \left( u( i_{\gamma}^{} \left([t_{1},t_{2}]\cap O \right))\right)
+ \Ll^{1}(N)
\\
&\!\!\stackrel{\eqref{E:loclipg}}{\leq}
G	\Ll^{1}\left([t_{1},t_{2}]\cap O\right) + 0
\\
&\leq
G(t_{2}-t_{1})
.\qedhere
\end{align*}
\end{proof}

\begin{remark}
If $u$ is $1/2$-H\"older continuous, we see from~\eqref{E:Dafermoscomputation1} 
that $u$ is Lipschitz along characteristics independently of any assumptions on inf{}lection points of $f$.
\end{remark}

\subsection{Construction of a universal source}
\label{Ss:univSource}
We now assume the negligibility of inf{}lection points~\eqref{E:nonvanishingCondition}.
Under this assumption, we generalize~\cite[\S~6]{BCSC} and we construct for general fluxes satisfying~\eqref{E:nonvanishingCondition} a source term for the broad formulation, without discussing its compatibility with the distributional source.
Namely, we show that
\begin{subequations}
\begin{align}
\label{E:secondaUniversal}
&\exists\gfr\in\pw(\Omega)
\ :
&&\ddt u(i_{\gamma}^{}(t))
=
\gfr(i_{\gamma}^{}(t))
&&\text{in $\D(i_{\gamma}^{-1}(\Omega))$ $\forall$characteristic curve $\gamma$.}
\end{align}
The compatibility of the sources will be instead matter of \cite[\S~3]{file2ABC}, where we prove that when inf{}lection points are negligible there is a choice of such $\gfr$ so that moreover
\begin{align}
\label{E:primaUniversal}
&\pt u(t,x) + \px (f(u(t,x))) =  \gfr(t,x)
&&\text{in $\D(\Omega)$.}
\end{align}
\end{subequations}
We deal here in \S~\ref{Ss:univSource} only with the ODE property~\eqref{E:secondaUniversal}.
We call such $\gu=[\gfr]_{\lambda}$ universal source term, and~\eqref{E:secondaUniversal} shows that $u$ is a broad solution of~\eqref{E:basicPDE}.
We mention nevertheless that if $f$ is not $\alpha$-convex, $\alpha>1$, then~\eqref{E:secondaUniversal} does not identify in general a distribution, because there can be an $\Ll^{2}$-positive measure set of points where $u$ is not differentiable along characteristics: in this set the proper definition of the source will come from~\eqref{E:primaUniversal}.

Two remarks before starting. Owing to \S~\ref{Ss:Lipregalongchar}, under the sharp vanishing condition~\eqref{E:nonvanishingCondition} on inf{}lection points of $f$ one gains $G$-Lipschitz continuity along characteristic curves for any continuous distributional solution $u$ to the balance law                                                                                                                                                                                                         \begin{equation*}
\tag{\ref{E:basicPDE}}
\pt u(t,x) + \px (f(u(t,x))) =  \gd(t,x)
\qquad f\in C^{2}(\R),
\qquad 
|\gd(t,x)|\leq G.
\end{equation*}
It is not of course possible to require that the reduction of the balance law on characteristics is satisfied for every $\gfr\in\pw(\R^{+}\times\R)$ such that $\gd=[\gfr]$, because altering $\gfr(t,x)$ on a curve provides the same distribution $[\gfr]$: this is why we need to select a good representative.
Without the negligibility~\eqref{E:nonvanishingCondition} the source term of a Lagrangian parameterization might not work with a different Lagrangian parameterization and there may exist no broad solution, see~\cite[\S~4.3]{file2ABC}.

We assume therefore the negligibility of inf{}lection points~\eqref{E:nonvanishingCondition} and we proceed as follows:
\begin{itemize}
\item[\S~\ref{Sss:souslSel}:] We construct a \emph{Souslin} function $\gfr$, which intuitively must satisfy~\eqref{E:secondaUniversal}.
\item[\S~\ref{Sss:borelSel}:] We construct an analogous \emph{Borel} function $\hat \gfr$, which is stronger but more technical.
\item[\S~\ref{Sss:proofSel}:] We prove that the functions $\gfr$ and $\hat \gfr$ do satisfy~\eqref{E:secondaUniversal}.
\end{itemize}
The construction for the compatibility condition~\eqref{E:primaUniversal} comes in \cite[\S~3
]{file2ABC}.

\subsubsection{Souslin selection}
\label{Sss:souslSel}
This is the first idea: to define pointwise, but in a measurable way, a function $\hat g(t,x)$ such that $t$ is a Lebesgue point for the derivative of the composition $t\mapsto u(t,\gamma(t))$, with $\gamma$ a characteristic function through $(t,x)$, whenever there exists one satisfying this differentiability property.
As we just consider the derivative of this composition at $t$ fixed, we focus on the curve only in a neighborhood of $t$ and, for notational convenience, we translate its domain to a neighborhood of the origin. 
Therefore, fixed some $\delta>0$, one applies a selection theorem to the subset $\mathcal G$ of
\begin{subequations}
\begin{align}
[\delta,+\infty)\times\R\times C^{1}([-\delta,\delta];\R)\times [-G,G]
&&\supset
&&\mathcal G
&&\ni
&&(t,x,\gamma,\zeta)
\end{align}
defined by the intersection among the set
\begin{itemize}
\item $\mathcal C$ of time-translated characteristics through $(t,x)$, in~\eqref{E:mathcalC},
\item $\mathcal D$ in~\eqref{E:mathcalD} where one imposes the pointwise differentiability at time $t$:
\end{itemize}
\begin{align}
\notag\mathcal G^{}=&\mathcal C\cap\mathcal D\\
\label{E:mathcalC}
=&\bigg\{(t,x,\gamma,\zeta)\quad :\quad 
\gamma(0)=x,
\quad  \dot\gamma(s) =\lambda(t+s,\gamma(s))\bigg\}\\
\label{E:mathcalD}
&\bigcap\bigg\{(t,x,\gamma,\zeta)\quad :\quad\zeta=\lim_{\sigma\downarrow 0}\frac{u(t\pm\sigma,\gamma(\pm\sigma))-u(t,x)}{\pm \sigma}
\bigg\}.
\end{align}
\end{subequations}
We first need a technical but important lemma about $\mathcal G$. The selection theorem will follow.

\begin{lemma}
\label{L:GBorel}
$\mathcal G$ is Borel. 
\end{lemma}
\begin{proof}
Focus first on the components $(t,x,\gamma)$.
The set $\mathcal C$ is closed thanks to the continuity of $\lambda$.

We discretize the limit in the variable $\zeta$, so that $\mathcal D$ is described as a $F_{\sigma\delta}$-set.

\begin{claim}
\label{C:discrlim}
Existence and the values of the following two limits are the same:
\begin{align}
&\lim_{h\downarrow 0}\frac{u(t+h,\gamma(h))-u(t,x)}{h}
\\
\label{E:sequenceh}
&\lim_{n\to\infty}\frac{u(t+h_{n},\gamma(h_{n}))-u(t,x)}{h_{n}},
&& h_{n+1}=h_{n}-h_{n}^{2},
&& h_{1}=1/2.
\end{align}
One can similarly have the full limit for $h\to0$ instead of $h\downarrow0$, that we study for simplicity.
\end{claim}
\begin{proof}[Proof of Claim~\ref{C:discrlim}]
By Theorem~\ref{T:sharpLipschitzreg} $u$ is $\norm{g}_\infty$-Lipschitz continuous on characteristic curves. Setting $(t,x)=(0,0)$ for notational convince, then for every $h\in(h_{n+1},h_{n}]$\begin{align*}
\bigg|
&\frac{ u(h,\gamma(h))-u(0,0)}{h} - \frac{ u(h_n,\gamma(h_n))-u(0,0)}{h_{n}}
\bigg|
\\
&=
\bigg|
\left(\frac{1}{h}-\frac{1}{h_{n}}\right)\big[ u(h,\gamma(h))-u(0,0)\big] 
	- \frac{1}{h_{n}}\big[ u(h_n,\gamma(h_n))-u(h,\gamma(h))\big]
\bigg|
\\
&\leq
2G\frac{h_{n}-h}{ h_{n}} .
\end{align*}
By construction however
\[
|h_{n}-h|
\leq
|h_{n}-h_{n+1}|
=
h_{n}^{2},
\]
yielding that the existence of the limit along $\{h_{n}\}_{n}$ implies the existence of the limit for any $h\downarrow0$.
\end{proof}

Notice that the claim would not hold choosing a generic $\tilde h_n\downarrow 0$ instead of $\{h_n\}_n$.
After observing that the limit is discrete, the classical differentiability constraint in~\eqref{E:mathcalD} is
\[
\forall k\ \exists n\ \forall \bar n\geq n:
\qquad
\left| \zeta-\frac{u(t\pm h_{\bar n}, \gamma(h_{\bar n}))-u(t,x)}{\pm h_{\bar n}}\right| \leq 2^{-k}.
\]
Therefore, $\mathcal D$ is equivalently defined as the $F_{\sigma\delta}$ set
\[
\mathcal D
=\bigcap_{k\in\N} \bigcup_{n\in\N} \bigcap_{\bar n\geq n} 
\bigg\{(t,x,\gamma,\zeta)\quad :\quad\left| \zeta-\frac{u(t\pm h_{\bar n}, \gamma(h_{\bar n}))-u(t,x)}{\pm h_{\bar n}}\right| \leq 2^{-k}
\bigg\}	
.
\]
Since the set within brackets is closed, $\mathcal G=\mathcal C\cap\mathcal D$ is Borel.
\end{proof}

Let $E$ be the projection of $\mathcal G$ on the first two components $\R^{+}\times \R$.\label{firstdefofE}
The set $E$ is the set of points where there exists an absolutely continuous (time-transalted) characteristic curve having $0$ as a density point for the derivative of $u(t+s,\gamma_{(t,x)}(s))$.
The selection theorem below assigns to every point $(t,x)$ where possible, which is to every point in $E$, an absolutely continuous integral curve $\gamma_{(t,x)}(s)$ for the ODE $\dot\gamma_{(t,x)}(s)=\lambda(t+s, \gamma_{(t,x)}(s))$ together with the Souslin function
\begin{equation}
\label{E:universalSource}
(t,x)\mapsto \gfr(t,x)\equiv\frac{d}{ds} u(t+s,\gamma_{(t,x)}(s))\Big|_{s=0}.
\end{equation}

\begin{remark}
\label{R:Sfullmeasure}
We comment on what information on $E$ comes from hypothesis on $f$:
\begin{enumerate}
\item If $f$ is $\alpha$-convex, we will observe in \cite[\S~3.2]{file2ABC} that the projection $E$ of $\mathcal G$ on $\R^+\times\R$ has full measure.
This follows for the case of quadratic f{}lux by a Rademacher theorem in the context of the Heisenberg group~\cite{FSSC,BCSC}.
\item If $f$ is even strictly but not uniformly convex~\cite[\S~4.2]{file2ABC} shows that $E$ may fail to have full measure. If (the closure of) inf{}lection points of $f$ are negligible, however, the Lipschitz continuity of $u$ along characteristics of Theorem~\ref{T:sharpLipschitzreg} implies that
\[
\Ha^{1}( i_{\gamma}(\R)\setminus E)=0
\]
for every characteristic curve $\gamma(t)$.
\item For general f{}luxes not only $E$ may not have full $\Ll^{2}$-measure, but also $i_{\gamma}(\R)\setminus E$ may not have full $\Ha^{1}$-measure for some characteristic curve $\gamma$ along which $u$ is \emph{not} Lipschitz-continuous, see~\cite[\S~4.3
]{file2ABC}.
\end{enumerate}
The set $E$ is considered also in~\cite[\S~3]{file2ABC} for the compatibility of the source terms.
\end{remark}

\begin{corollary}[Selection theorem]
\label{S:selectionth}
For every $\delta>0$, there exists a function
\begin{align*}
[\delta,+\infty)\times \R\supset E&&\ni&&(t,x)&&\mapsto &&(\gamma_{(t,x)}(s), \gfr(t,x))&& \in &&C^{1}([-\delta,\delta];\R)\times[-G,G]
\end{align*}
which is measurable for the $\sigma$-algebra generated by analytic sets and which satisfies by definition
\begin{align*}
(t,x,\gamma_{(t,x)}(s), \gfr(t,x))\in\mathcal G.
\end{align*}
\end{corollary}
\begin{proof}
The Borel measurability of $\mathcal G$ proved in Lemma~\ref{L:GBorel} allows to apply to $\mathcal G$ Von Neumann selection theorem~\cite[Theorem~5.5.2]{Sri}, from~\cite{VN}, which provides the thesis.
\end{proof}
\begin{definition}
We define as a Souslin universal source the function
\[
 \gfr(t,x)=\gfr_{t,x}=
\begin{cases}
0 &(t,x)\notin E \\
0 & u(t,x)\in N\\
 \gfr_{t,x} &(t,x)\in E,\ u(t,x)\notin N  .
\end{cases}
\]
\end{definition}
The importance of the above selection theorem is due to the following relation.
\begin{theorem}
\label{T:univselection}
Assume that $\Ll^{1}(\clos({\infl(f)}))=0$.
Then for every absolutely continuous integral curve $\gamma$ of the ODE $\dot\gamma=\lambda(i_{\gamma})$, one has that $\gfr (i_{\gamma})$ is well defined $\Ll^{1}$-a.e.~and it satisfies
\[
u(i_{\gamma}(s))-u(i_{\gamma}(r))=\int_{r}^{s}  \gfr (i_{\gamma}(t)) dt
\qquad \forall 0\leq r\leq s.
\]
\end{theorem}
Theorem~\ref{T:univselection} is fairly not trivial because in~\eqref{E:universalSource} the universal source $ \gfr(t,x)$ is defined as the derivative of $u$ along a chosen curve $\gamma_{(t,x)}$ which changes changing the point $(t,x)$, and it is not even defined on a full measure set!
What is relevant for the theorem is that the set where $\gfr$ is not defined, or not uniquely defined, is negligible along any characteristic curve, which is that
\[
\gu=[ \gfr]_{\lambda}
\]
is well defined independently of the selection we have made. Different selections may change $ \gfr$, but not $\gu$.
We postpone the proof of the theorem and of this fact to \S~\ref{Sss:proofSel}, after showing that it is possible to define a Borel selection $\hat\gfr$.
Theorem~\ref{T:univselection} implies that $\gfr$ and $\hat\gfr$ give the same $\gu$.

\subsubsection{Borel selection}
\label{Sss:borelSel}
Before proving Theorem~\ref{T:univselection}, for the sake of completeness we show that one can define as well a Borel function, that we denote by $\hat \gfr(t,x)=\hat\gfr_{t,x}$, for which Theorem~\ref{T:univselection} still holds. This requires a bit more work than the previous argument: we do not associate immediately to each point (where it is possible) an eligible curve and the derivative of $u$ along it, but something which must be close to it.
We find then with the proof of Theorem~\ref{T:univselection} that we end up with the same class $\gu=[\hat \gfr]_{\lambda}$.

\begin{lemma}
\label{L:BorelApprg}
The $(t,x)$-projection $E$ of $\mathcal G$ is Borel. For every $\varepsilon>0$ there exists a Borel function
\begin{align*}
S&&\ni&&(t,x)&&\mapsto &&(\gamma_{\varepsilon,t,x},\gfr_{\varepsilon,t,x}) &&\in &&C^{1}([-\delta,\delta])\times[-G-\varepsilon,G+{\varepsilon}]
\end{align*}
such that $(t,x,\gamma_{\varepsilon,t,x},\gfr_{\varepsilon,t,x}) \in\mathcal C$ of~\eqref{E:mathcalC} and such that for $|h|$ sufficiently small
\[
\left| \gfr_{\varepsilon,t,x} - \frac{u(t\pm h,\gamma_{\varepsilon,t,x}(\pm h))-u(t,x)}{\pm h}\right|<\varepsilon .
\]
\end{lemma}

\begin{definition}
\label{D:Borelg}
We define as a Borel universal source the function
\[
\hat \gfr(t,x)=\hat\gfr_{t,x}=
\begin{cases}
0 &(t,x)\notin E \\
0 & u(t,x)\in N\\
\liminf_{\varepsilon\downarrow 0} \gfr_{\varepsilon,t,x} &(t,x)\in E,\ u(t,x)\notin N  ,
\end{cases}
\]
where $\gfr_{\varepsilon,t,x}$ is fixed in Lemma~\ref{L:BorelApprg}.
\end{definition}
\begin{proof}[Proof of Lemma~\ref{L:BorelApprg}]
We remind the following selection theorem~\cite[Th.~5.12.1]{Sri}.
\begin{theorem}(Arsenin-Kunugui)
Let $B \subset X \times Y$ be a Borel set, $X, Y$ Polish, such that $B_{x}$ is $\sigma$-compact for every $x$. Then the projection on $X$ of $B$ is Borel, and $B$ admits a Borel function $s:P_{X}B\to Y$ such that $(x,s(x))\in B$ for all $x$ in the projection $P_{X}B$.
\end{theorem}

We verify the hypothesis of and we apply the above selection theorem to the set
\begin{equation}
\label{E:accessorio2}
\bigcup_{n\in\N} \bigcap_{m>n}
\left\{
(t,x,\gamma_{},\zeta) \in \mathcal C:
\ \left| \zeta -  \frac{u(t\pm h_{m},\gamma(\pm h_{m}))-u(t,x)}{\pm h_{m}}\right|\leq \frac{\varepsilon}{2}
\right\},
\end{equation}
where $\mathcal C$ was defined in~\eqref{E:mathcalC} and $\{h_{n}\}_{n\in\N}$ immediately below that in~\eqref{E:sequenceh}.
The section
\[
\mathcal C_{(\bar t,\bar x)}=\{(\bar t,\bar x, \gamma,\zeta)\quad :\quad 
\gamma(0)=\bar x,
\quad  \dot\gamma(s) =\lambda(t+s,\gamma(s))\}
\] 
is locally compact as a consequence of Ascoli-Arzel\`a theorem, because by the boundedness of $\lambda$ the curves are equi-bounded and equi-Lipschitz continuous, and by the continuity of $\lambda$ when they converge uniformly they also converge in $C^{1}([-\delta,\delta])$.
For $t,h$ fixed the set
\[
\left\{
(\gamma,\zeta):\ 
\left| \zeta -  \frac{u(t\pm h,\gamma(\pm h))-u(t,x)}{\pm h}\right|\leq\frac{\varepsilon}{2}
\right\}
\]
is closed, therefore its intersection with $\mathcal C_{(\bar t,\bar x)}$ is compact: this proves that each $(t,x)$-section of~\eqref{E:accessorio2} is $\sigma$-compact.
The hypothesis of the theorem are satisfied: it provides that the projection $E$ of~\eqref{E:accessorio2} on the first factor $\R^+\times\R\ni(t,x)$ is Borel and that there exists a Borel subset of~\eqref{E:accessorio2} which is the graph of a function $s$ defined on $E$. Since the projection from that graph to the first components $\R^{+}\times\R$ is one-to-one and continuous, the function $s$ is a Borel section of~\eqref{E:accessorio2}, concluding our statement.
\end{proof}

\subsubsection{Proof of Theorem~\ref{T:univselection}}
\label{Sss:proofSel}
We provide here the proof of Theorem~\ref{T:univselection} with $\gfr_{t,x}$ either the Borel or the Souslin one.
Let us introduce the notation.
We consider:
\begin{itemize}
\item $\bar \gamma(s)$ a characteristic curve for the balance law through a point $(t,x)=(t,\bar\gamma(t))$.
\item $\bar\zeta(s)=\dds u(s, \bar\gamma(s))$ the derivative of $u$ along $\bar \gamma$, where it exists.
\item Either $\gamma_{\tilde \varepsilon,t, x}(s)$ or $\gamma_{t, x}(s)$: the characteristic curve of $u$ through $(t, x)$ given either by Lemma~\ref{L:BorelApprg} or by Corollary~\ref{S:selectionth}. Fix for example $\gamma_{\tilde \varepsilon,t,x}(s)$, which is more complex.
\item $\zeta_{\tilde\varepsilon}(t+s)=\dds u(t+s,\gamma_{\tilde \varepsilon,t, x}(s))$ the derivative of $u$ along $\gamma_{\tilde \varepsilon,t, x}$, where it exists. Where the derivative does not exists, set for example the function equal to $0$.
\item Either $\hat \gfr({t,x})$ of Definition~\ref{D:Borelg} or $\gfr(t,x)$ of Corollary~\ref{S:selectionth}. Fix $\hat \gfr({t,x})$, as we are showing the proof with the Borel selection of \S~\ref{Sss:borelSel}.
\end{itemize}
We indeed know from Theorem~\ref{T:sharpLipschitzreg} that $u(i_{\bar \gamma}(s))$ and $u(i_{\gamma_{\tilde \varepsilon,t, x}}(s))$ are $G$-Lipschitz  continuous.
We prove first that for almost every $t$ the derivative of $u(i_{\bar\gamma}(t))$ is precisely $\hat \gfr(i_{\bar\gamma}(t))$ if $u(i_{\bar\gamma}(t))$ is not an inf{}lection point of $f$. 
After that, we exploit again the negligibility assumption $\Ll^{1}(\clos({\infl(f)}))=0$ on the inf{}lection points of $f$ and we conclude 
\[
u(i_{\gamma}(s))-u(i_{\gamma}(r))=\int_{r}^{s} \hat \gfr (i_{\gamma}(t)) dt
\qquad \forall 0\leq r\leq s.
\]
\firststep
\step{Countable decomposition} We give a countable covering of the set of Lebesgue points $t$ where the derivative $\bar\zeta(t)$ of $u(i_{\bar\gamma}(t))$ exists but it differs from $\hat\gfr_{i_{\bar\gamma}(t)}$. The set can be described as
\begin{align*}
&\bigcup_{\varepsilon\downarrow 0}
\left\{
t:\ 
\exists \bar \zeta(t)=\lim_{\sigma\to 0}\frac{u(i_{\bar\gamma}(t+\sigma))-u(i_{\bar\gamma}(t))}{\sigma},
\quad
\left| \hat\gfr_{i_{\bar\gamma}(t)} - \bar \zeta(t)\right|\geq\varepsilon
\right\}.
\end{align*}
In particular, dropping the condition that the derivative $\bar\zeta(t)$ of $u(i_{\bar\gamma}(t))$ exists at $t$ we notice that this set is contained in
\begin{align*}
&
\bigcup_{\varepsilon\downarrow 0}
\bigcup_{n\in\N}
\left\{
t:\ \forall\sigma\in(0,2^{-n}) \quad
\left|\hat\gfr_{i_{\bar\gamma}(t)} - \frac{1}{\sigma} \int_{t}^{t+ \sigma}\bar\zeta \right|\geq \varepsilon
, \quad 
\left|\hat\gfr_{i_{\bar\gamma}(t)} -
\frac{1}{\sigma} \int_{t-\sigma}^{t}\bar\zeta \right|\geq \varepsilon
\right\} 
\end{align*}
The proof now needs a further index because we are working with the Borel selection.
If we remember the Definition~\ref{D:Borelg} of $\hat \gfr_{t,x}$, and we observe that along the characteristic curve $\gamma_{\tilde \varepsilon,t, \bar\gamma(t)}$
\[
u(i_{\gamma_{\tilde \varepsilon,t, x}}(t\pm\sigma))-u(i_{\gamma_{\tilde \varepsilon,t, x}}(t))=\int_{t}^{t\pm \sigma}\zeta_{\tilde\varepsilon},
\] 
then one can add a condition which is always satisfied and the last union can be rewritten as 
\begin{align*}
&
\bigcup_{\tilde \varepsilon<\varepsilon\downarrow 0}
\bigcup_{n\in\N}
\left\{
t:\ \forall \sigma \in(0,2^{-n}) \quad 
\left| \hat\gfr_{i_{\bar\gamma}(t)} - 
\frac{1}{\pm \sigma} \int_{t}^{t\pm \sigma}\bar\zeta\right|\geq 3\varepsilon
, \quad
	\left|\gfr_{\tilde \varepsilon,i_{\bar\gamma(t)}} 
		- \frac{1}{\pm \sigma} \int_{t}^{t\pm \sigma}\zeta_{\tilde\varepsilon}\right|
		< \varepsilon
\right\} .
\end{align*}
The union can as well be done on any sequences $\tilde\varepsilon_{k}<\varepsilon_{k}\downarrow0$: if $\left|\hat\gfr_{t,x}-\gfr_{\tilde \varepsilon_{k},t, x}\right|\leq \varepsilon_{k}/3$ then one has the equivalent expression
\begin{align*}
\bigcup_{\tilde\varepsilon_{k}<\varepsilon_{k}\downarrow 0}
\bigcup_{n\in\N}
\left\{
t:\ \forall \sigma \in(0,2^{-n}) \quad 
\left|  \hat\gfr_{i_{\bar\gamma}(t)} - 
\frac{1}{\pm \sigma} \int_{t}^{t\pm \sigma}\bar\zeta\right|\geq 3\varepsilon_{k}
, \quad
	\left| \hat\gfr_{i_{\bar\gamma}(t)}
		- \frac{1}{\pm \sigma} \int_{t}^{t\pm \sigma}\zeta_{\tilde\varepsilon_{k}}\right|
		<\varepsilon_{k}
\right\} .
\end{align*}
We arrived to the countable covering that we wanted to prove in this step.

If one is considering the Souslin selection clearly
$\gamma_{\tilde \varepsilon,t,x}=\gamma_{t,x}$ and $\gfr_{\tilde \varepsilon,t,\bar\gamma(t)}=\gfr_{t,\bar\gamma(t)}=\zeta_{\tilde\varepsilon_{k}}(t)$.

\step{Reduction argument}
We prove that the set
\begin{equation}
\label{E:accessorio}
\left\{
t:\ \forall\sigma\in (0,2^{-n}) \quad 
 \hat\gfr({i_{\bar\gamma}(t)}) > 
\frac{1}{\pm \sigma} \int_{t}^{t\pm \sigma}{\bar\zeta}+3\varepsilon
, \quad
	\left| \hat\gfr({i_{\bar\gamma}(t)})
		- \frac{1}{\pm\sigma} \int_{t}^{t\pm \sigma} \zeta_{\tilde\varepsilon}\right|
		<\varepsilon
\right\} 	
\end{equation}
cannot contain two points $t_{1}, t_{2}$ with $|t_{1}-t_{2}|\leq 2^{-n}$.
The case
\[
 \hat\gfr({i_{\bar\gamma}(t)})<
\frac{1}{\pm \sigma} \int_{t}^{t\pm \sigma}{\bar\zeta}-3\varepsilon 
\]
is similar, backwards in time.
Then the thesis will follow: by the previous step, the set of times where the derivative of $u(i_{\bar\gamma}(t))$ exists and it is different from $\hat\gfr_{i_{\bar\gamma}(t)}$ will be at most countable. Therefore the derivative of $u(i_{\bar\gamma}(t))$ will be almost everywhere precisely $\hat\gfr_{i_{\bar\gamma}(t)}$.

\step{Analysis of the single sets}
By contradiction, assume that~\eqref{E:accessorio} contains two such points, for example $t_{1}=0$, $t_{2}=\rho$.
Then, essentially two cases may occur.

\firstsubstep
\substep{Concavity/convexity region}
We first consider the open region where $f''(u)\geq 0$. The open region $f''(u)\leq0$ is entirely similar. The restriction to the open set is allowed because the argument is local: we consider later also the region of inf{}lection points. 
In particular, in this step we consider $f'(u)$ monotone in $u$, in particular nondecreasing. 

Compare $\bar \gamma$ with the two curves given by the selection theorem through two fixed points
\[
(0,\bar \gamma(0))=(0,0), \qquad (\rho,\bar\gamma(\rho))=(\rho,0)
\]
The $x$ component is set $0$ just for simplifying notations.
We rename the characteristic curves as
\begin{gather*}
\gamma_{0}(t):=\gamma_{\tilde\varepsilon,0, \bar\gamma(0)}(t), \quad
\zeta_{0}(t):=\ddt u(i_{\gamma_{0}}(t)),\\
\gamma_{\rho}(t):=\gamma_{\tilde\varepsilon, \rho, \bar\gamma(\rho)}(t-\rho),
\quad
\zeta_{\rho}(t):=\ddt u(i_{\gamma_{\rho}}(t))
.
\end{gather*}
Notice that $\gamma_{0}(t)$, $\gamma_{\rho}(t)$ are tangent to $\bar\gamma$ respectively at times $0$, $\rho$ because they are characteristics.
By~\eqref{E:accessorio}, at respectively $t=0$, $t=\rho$, one finds for $t\in[0,\rho]$
\begin{align*}
\frac{1}{\pm\sigma}\int_{t}^{t\pm\sigma}\zeta_{0} \geq \hat\gfr({i_{\bar\gamma}(t)})-\varepsilon \geq \frac{1}{\pm\sigma}\int_{t}^{t\pm\sigma}\bar \zeta+2\varepsilon
\\
\frac{1}{\pm\sigma}\int_{t}^{t\pm\sigma}\zeta_{\rho} \geq \hat\gfr({i_{\bar\gamma}(t)})-\varepsilon \geq \frac{1}{\pm\sigma}\int_{t}^{t\pm\sigma}\bar \zeta+2\varepsilon
\end{align*}
which means that the derivative $\bar\zeta$ of $u$ along $\bar \gamma$ is lower than the ones $\zeta_{0},\zeta_{\rho}$ along $\gamma_{0}$, $\gamma_{\rho}$:
\begin{gather*}
 u(i_{\gamma_{0}}(t))-u(0,0)=\qquad\int_{0}^{t} \zeta_{0}\stackrel{\eqref{E:accessorio}}{\geq} \int_{0}^{t}\bar \zeta \qquad= u(i_{\bar\gamma_{}}(t))-u(0,0)
\\
u(\rho,0)- u(i_{\gamma_{\rho}}(t))=\qquad\int_{t}^{\rho} \zeta_{\rho}\stackrel{\eqref{E:accessorio}}{\geq} \int_{t}^{\rho}\bar \zeta \qquad= u(\rho,0)
-u(i_{\bar \gamma_{}}(t)).
\end{gather*}
This means that for $t$ in $(0,\rho)$ one has
\begin{align*}
&u(t,\gamma_{0}(t))\geq u(t,\bar\gamma(t)),
&&u(t,\bar\gamma(t))\geq u(t,\gamma_{\rho}(t)).
\end{align*}
Being $f'(u)$ nonincreasing in turn
\[
f'(u(t,\gamma_{\rho}(t))) \leq f'(u(t,\bar\gamma(t))) \leq f'(u(t,\gamma_{0}(t))) . 
\]
Being characteristics, the functions above are just the slopes of the curves $\gamma_{\rho}$, $\bar\gamma$, $\gamma_{0}$: integrating 
\begin{itemize}
\item $\bar\gamma$, $\gamma_{0}$ between $0$, where they coincide, and $t$ 
\item $\gamma_{\rho}$, $\bar\gamma$ between $\rho$, where they coincide, and $t$ 
\end{itemize}
one obtains
\[
\bar\gamma(t) \leq \gamma_{0}(t),
\qquad
\bar\gamma(t) \leq \gamma_{\rho}(t).
\]
As a consequence of this and of the finite speed of propagation, $\gamma_{0}$ and $\gamma_{\rho}$ must intersect in the time interval $[0, \rho]$, say at time $\rho'$.
We can compute the value of $u$ at 
\[
(\rho', \gamma_{0}(\rho'))=i_{\gamma_{0}}(\rho')=i_{\gamma_{\rho}}(\rho')=(\rho', \gamma_{\rho}(\rho'))
\]by the differential relation both on $\gamma_{0}$, starting from $0$, and on $\gamma_{\rho}$, starting from $\rho$: we have then
\[
u(0,\gamma(0)) + \int_{0}^{\rho'} \zeta_{0}
= u(\rho', \gamma_{0}(\rho'))
= u(\rho', \gamma_{\rho}(\rho'))
=
u(\rho, \gamma_{\rho}(\rho)) - \int_{\rho'}^{\rho}\zeta_{\rho}.
\]
Comparing the LHS and the RHS, one deduces
\[
u(\rho, \gamma_{\rho}(\rho)) -u(0,\gamma(0))
=
\int_{0}^{\rho'} \zeta_{0} + \int_{\rho'}^{\rho}\zeta_{\rho}
.
\]
However, the times $0$, $\rho $ belong by construction to the set~\eqref{E:accessorio}: therefore
\[
\int_{0}^{\rho'}\zeta_{0} + \int_{\rho'}^{\rho}\zeta_{\rho}
>
\rho' \, \hat\gfr({i_{\bar\gamma}(0)}) +  (\rho-\rho')\, \hat\gfr({i_{\bar\gamma}(\rho)})- 2\varepsilon
>
\int_{0}^{\rho} {\bar\zeta} + \varepsilon  .
\]
Since the RHS is just $u(\rho, \gamma_{\rho}(\rho)) -u(0,\gamma(0))+\varepsilon
$, we reach a contradiction.

\substep{Inf{}lection points}
The previous point proves the statement in the connected components of $u^{-1}(\R\setminus N)$, where $N=\clos(\infl(f))$.
The assumption $\Ll^{1}(N)=0$ allows to show that inf{}lection points do not matter. Indeeed, by Theorem~\ref{T:sharpLipschitzreg} the composition $U=u\circ i_{\gamma}$ is $G$-Lipschitz continuous: Lemma~\ref{L:dersuinsieminulli} below assures therefore that $u(i_{\gamma})$ is differentiable with $0$ derivative $\Ll^{1}$-a.e.~on $(u\circ i_{\gamma})^{-1}(N)$.
For every $r<s$, by the previous half point
\begin{align}
\notag
u(i_{\gamma}(s))-u(i_{\gamma}(r))&= \int_{r}^{s}\zeta(t)dt
\\
\notag
&=\int_{[r,s]\cap (u\circ i_{\gamma})^{-1}(N)}\zeta(t)dt
+\int_{[r,s]\setminus (u\circ i_{\gamma})^{-1}(N)}\zeta(t)dt
\\
\notag
&=\int_{[r,s]\cap (u\circ i_{\gamma})^{-1}(N)}0dt
+\int_{[r,s]\setminus (u\circ i_{\gamma})^{-1}(N)}
\hat \gfr (i_{\gamma}(t))dt
\end{align}
Remember that $\hat \gfr=0$ on $u^{-1}(N)$ by definition. This yields the thesis of Theorem~\ref{T:univselection}:
\[
u(i_{\gamma}(s))-u(i_{\gamma}(r))=\int_{r}^{s} \hat \gfr (i_{\gamma}(t)) dt
\qquad \forall 0\leq r\leq s.
\]

\begin{lemma}
\label{L:dersuinsieminulli}
Consider a Lipschitz continuous function $U:\R\to\R$ and a Lebesgue negligible set $N\subset\R$. Then the derivative of $U$ vanishes $\Ll^{1}$-a.e.~on $U^{-1}(N)$.
\end{lemma}
\begin{proof}
Let $T_{N}$ be the set of Lebesgue points of $U^{-1}(N)$: 
\begin{equation*}
T_{N}=\left\{\bar t\ :\ \lim_{h\downarrow 0}\frac{\Ha^{1}\Big([\bar t-h, \bar t+h]\times \{\bar \ytau\} \setminus  U^{-1}(N)\Big)}{h}=0\right\}.
\end{equation*}
If $U$ is $G$-Lipschitz continuous, then for $\bar t\in T_{N}$ one has
\begin{align*}
\frac{|U{}(\bar t+h)-U{}(\bar t)|}{h}
\leq&
\frac{\Ll^{1}(U{}([\bar t,\bar t+h])\cap N)}{h} + \frac{\Ll^{1}(U{}([\bar t,\bar t+h])\setminus N)}{h}
\\
\leq&
0 + \frac{\Ll^{1}(U{}([\bar t,\bar t+h])\setminus N)}{h}
\\
\leq&
G \frac{\Ll^{1}([\bar t,\bar t+h]\times\{\bar\ytau\}\setminus U^{-1}(N))}{h}.
\end{align*}
The RHS converges to $0$ as $h\to 0$ because $t\in T_{N}$.
This shows that $u\circ i_{\gamma}(t)$ is differentiable at every $\bar t\in T_{N}$ with $0$ derivative: this concludes the proof of the lemma because $\Ll^{1}$-a.e.~point of any Lebesgue measurable subset of $\R$ is a Lebesgue point of the set.
\end{proof}

\section{Distributional solutions are Lagrangian solutions}
\label{S:distrareLagrangian}

Consider a continuous distributional solution of
\begin{equation*}
\tag{\ref{E:basicPDE}}
\pt u(t,x) + \px (f(u(t,x))) =  \gd(t,x)
\qquad f\in C^{2}(\R),
\qquad 
|\gd(t,x)|\leq G
\end{equation*}
When inf{}lection points of $f$ are negligible, $u$ is Lipschitz continuous along any characteristic curve $\gamma$ of $u$ (Theorem~\ref{T:sharpLipschitzreg}).
If not, then we have cases when $u$ is not Lipschitz continuous along some characteristics~\cite[\S~4.3]{file2ABC}), and the points where $u$ may not be differentiable along \emph{any} characteristic curve might have positive $\Ll^{2}$-measure.

Here we work without the assumption on inf{}lection points.
We show first that continuous distributional solutions do not dissipate entropy (Lemma~\ref{L:nodissipation2}).
By approximation of the entropy, this reduces to the case of negligible inf{}lection points, where the solution is broad and therefore Lagrangian, and it exploits the consequent $\BV$-approximation of \S~\ref{S:lagraredistr}.

We show then in Lemma~\ref{L:Excharlip} that, given an \emph{entropy} continuous distributional solution $u$, one can find through each point a characteristic curve $\bar \gamma(t)$ along which $u$ is Lipschitz continuous.
As a consequence, by \S~\ref{Ss:lagrparam} one can construct a Lagrangian parameterization and deduce that $u$ is a Lagrangian solution.

\begin{lemma}
\label{L:nodissipation2}
Continuous distributional solutions of~\eqref{E:basicPDE} do not dissipate entropy.
\end{lemma}

\begin{proof}
If the closure $N$ of the inf{}lection points of $f$ is negligible, then by Theorem~\ref{T:univselection} a continuous distributional solution $u$ is a broad solution, and by Lemma~\ref{L:Elagrpar} it is in particular a Lagrangian solution. By Corollary~\ref{C:lagrnodiss}, derived from the monotone approximations of Lagrangian continuous solutions, one has then that $u$ satisfies the entropy equality.

If the inf{}lection points of $f$ are not negligible, one can derive the thesis by an approximation procedure.
Fist notice that for every entropy-entropy f{}lux pair---which means for every function $\eta\in C^{1,1}(\R)$ and every $q\in C^{1,1}(\R)$ satisfying $q'(z)=\eta'(z)f'(z)$---one has the entropy equality
\[
\pt\eta(u) + \px (q(u)) 
=
\eta'(u)\gd 
\qquad \text{in $\D(\R^{+}\times\R\setminus u^{-1}(N))$}
\]
by the previous step; indeed, in the open set $\R^{+}\times\R\setminus u^{-1}(N)$, where we are claiming that the PDE holds in the sense of distribution, $u$ is valued where $f$ does not have inf{}lection points and therefore one can apply Corollary~\ref{C:lagrnodiss}.

Consider finally a decreasing family of open sets $O_{k}=\cup_{j}(a_{j}^{k}, b_{j}^{k})\subset \R$ 
such that
\begin{itemize}
\item $N=\clos(\infl(f)) \subset  O_{k}$ for $k\in\N$;
\item $|O_{k}\setminus N|<1/k$.
\end{itemize}
One can approximate $\eta$ in $C^{1}(\R)$ with entropies $\eta_{k}\in C^{1,1}(\R)$ which are linear in $O_{k}$, for $k\in\N$.
For every interval $(a_{j}^{k},b_{j}^{k})\subset O_{k}$ where $\eta_{k}(u)=c_{k} u$ for some $c_{k}\in\R$, for all $k\in\N$ one has
\begin{align*}
\pt\eta_{k}(u) + \px (q_{k}(u)) 
= c_{k} \left[\pt u + \px f(u) \right]= c_{k}\gd =
\eta_{k}'(u)\gd 
\qquad \text{in $\D( u^{-1}((a_{j}^{k},b_{j}^{k})))$}
\\
\pt\eta_{k}(u) + \px (q_{k}(u)) 
=
\eta_{k}'(u)\gd 
\qquad \text{in $\D(\R^{+}\times\R\setminus u^{-1}(N))$}
\end{align*}
This shows that the entropy equality holds for the entropies $\{\eta_{k}\}_{k\in\N}$.
When $\eta_{k}(u), q_{k}(u), \eta'_{k}(u)$ converge uniformly to $\eta_{}(u), q_{}(u), \eta'_{}(u)$, then the entropy equality holds also for $\eta$.
\end{proof}

\begin{corollary}
\label{C:fprburg}
If $u$ is a continuous distributional solution of~\eqref{E:basicPDE}, then $f'(u)$ is a continuous distributional solution of Burgers' balance law with source term $f''(u)\gd$.
\end{corollary}

While Lemma~\ref{L:nodissipation2} above relies on the previous results of this paper, Lemma~\ref{L:Excharlip} below is instead self-contained.
It is however based on maximum principle, that we recall now. 

\begin{lemma}
\label{L:contmaxprinc}
Suppose $u, v$ are entropy solutions of the PDE
\begin{align*}
\pt u(t,x) + \px (f(u(t,x))) =  \gd_{u}(t,x)
\\
\pt v(t,x) + \px (f(v(t,x))) =  \gd_{v}(t,x)
\end{align*}
and that for some $G\in\R$
\[
u(t=0,x)\leq v(t=0,x)
\qquad
-G\leq \gd_{u}(t,x) \leq \gd_{v}(t,x)\leq G.
\]
Then $u(t,x)\leq v(t,x)$.
\end{lemma}

\begin{proof}
See the proof of Theorem 3,~Page 229,~\cite{Kruzkov}.
Alternative approaches are the vanishing viscosity or the operator splitting, still exploiting the uniqueness of the entropy solution.
\end{proof}

\begin{lemma}
\label{L:Excharlip}
Suppose $u$ is a continuous entropy solution of the PDE
\begin{equation*}
\tag{\ref{E:basicPDE}}
\pt u(t,x) + \px (f(u(t,x))) =  \gd(t,x)
\qquad f\in C^{2}(\R),
\qquad 
|\gd(t,x)|\leq G
\end{equation*}
Then at each point $(t,x)$ there exists a characteristic along which $u$ is Lipschitz continuous.
\end{lemma}

\begin{proof}
The proof is made constructing a piecewise affine approximation of the desired characteristic curve.
On the two consecutive edges of the linearized curve, the Lipschitz regularity holds by the maximum principle, being an entropy solution.

\firststep
\step{Notation}
We simplify the notation changing coordinates so that we are looking for a characteristics curve throughout the point $(0,0)$, and defined between the times $t=-1$, $t=1$.

As the construction is local, we directly fix a square
\[
Q=[-1,1]^{2}.
\]
Let $G=\norm{\gd}_{\distrBdd(Q)}$.
Set $L=\norm{f'(u)}_{\distrBdd(Q)} $ and $M=\norm{f''(u)}_{\distrBdd(Q)}$.
Notice that $L$ is an upper bound for the characteristic speed $\lambda=f'(u)$ in $Q$.
Assume e.g.~$L<1$.

\step{Modulus of continuity for $u$ and $\lambda$}
As $u$ is continuous, in the compact region
$Q$ it is uniformly continuous.
Let $\omega(\delta)$ denote the following modulus of continuity of $u$ in $Q$:
\[
\max_{Q}\left\{|t'-t|, \frac{|x'-x|}{L}\right\}\leq \delta
\quad\Rightarrow\quad
|u(t,x)-u(t',x')| \leq \omega\left(\delta\right).
\]
An analogous modulus of continuity for $\lambda=f'(u)$ is clearly given by $M\omega(\delta)$:
\[
|\lambda(t,x) - \lambda(t',x')|
=
|f'(u(t,x))-f'(u(t',x'))| 
\leq
M |u(t,x)-u(t',x')| 
\leq
M\omega(\delta).
\]

\step{Dependency regions}
Let $\delta>0$ and $(\bar t, \bar x)\in [-1+\delta,1-\delta]^{2}$. One draws a backward triangle of dependency for an interval of time $\delta$, delimited from below by the segment $\bar t-\delta$:
\[
T(\bar t,\bar x)=\left\{ (t,x)\ :\qquad \bar t-\delta \leq t \leq \bar t,\ \bar x- L(\bar t-t)  \leq x\leq \bar x + L(\bar t-t)  \right\}.
\]
Noticing that the speed of propagation $\lambda$ in the rectangle
\[
\left\{ (t,x)\ : \  \max\left\{|\bar t-t|, \frac{|\bar x-x|}{L}\right\} \leq \delta \right\}
\]
is, by definition of the modulus of continuity $\omega$, bounded by $\lambda(\bar t, \bar x)\pm M\omega(\delta)$, a smaller backward triangle of dependency is given by
\[
T_{\delta}(\bar t,\bar x)=\left\{ (t,x)\ :\ \bar t-\delta \leq t \leq \bar t,\quad  [\lambda(\bar t, \bar x)- M\omega(\delta)](\bar t-t)  \leq x-\bar x\leq [\lambda(\bar t, \bar x)+ M\omega(\delta)](\bar t-t)  \right\}
\]
The basis of $T_{\delta}$ has length $2 M\delta \omega(\delta)$, which is superlinear in $\delta$.

\step{Comparison of $u$ on adjacent nodes}
Let $\varepsilon>0$. The linear functions
\[
u^{+}(t,x)=u(\bar t,\bar x)+G(t-\bar t)-\varepsilon ,
\qquad
u_{-}(t,x)=u(\bar t,\bar x)-G(t-\bar t) +\varepsilon
\] 
satisfy both \[u^{+}(\bar t,\bar x)< u(\bar t, \bar x) < u_{-}(\bar t, \bar x)\] and the equations
\[
\pt u_{-}(t,x) + \px (f(u_{-}(t,x))) =  -G,
\qquad
\pt u^{+}(t,x) + \px (f(u^{+}(t,x))) =  G.
\]
If we had either $u_{-}( t,  x) < u( t,  x)$ or $u^{+}( t, x) > u^{}( t, x)$ for all $x$ belonging to a $\varepsilon$-neighborhood of the basis of the small backward triangle of dependency $T_{\delta}$, we would contradict the maximum principle in Lemma~\ref{L:contmaxprinc}.
Therefore there exists a point $x_{\varepsilon}$ belonging to
\[
\big(\bar x- [\lambda(\bar t, \bar x)+ M\omega(\delta)]\delta -\varepsilon ,\ \bar x+ [\lambda(\bar t, \bar x)+ M\omega(\delta)]\delta +\varepsilon\big),
\]
which is a $\varepsilon$-neighborhood of the basis of $T_{\delta}$, where $u(t,x_{\varepsilon})$ is between $u^{+}$ and $u_{-}$:
\[
u(\bar t,\bar x)+G\delta+\varepsilon
=u_{-}(\bar t-\delta,x_{\varepsilon})
\leq u(\bar t-\delta,x_{\varepsilon})\leq
u^{+}(\bar t-\delta,x_{\varepsilon})=u(\bar t,\bar x)-G\delta-\varepsilon .
\]
As $\varepsilon\downarrow 0$, a subsequence of $\{x_{\varepsilon}\}_{\varepsilon\downarrow0}$ must converge to a point $\hat x$ belonging to the basis of $T_{\delta}$:\begin{equation}
\label{E:uonbasis}
\begin{split}
\exists \hat x\in [\bar x+ (\lambda(\bar t, \bar x)- M\omega(\delta))\delta, \bar x+ (\lambda(\bar t, \bar x)- M\omega(\delta))\delta]
\quad:\\
|u(\bar t-\delta,\hat x) - u(\bar t,\bar x)| \leq G\delta
.
\end{split}
\end{equation}

\step{Piecewise approximation}
We construct here a piecewise affine approximation of a backward characteristic through $(0,0)$, specifying the nodes: for $k\in\N$ set
\[
(t_{k},x_{k})=(0,0), 
\quad\text{and let}\quad
(t_{i},x_{i}):=\left (  - \frac{k-i}{k},\ x_{i} \right)
\qquad
i=0,\dots ,k-1
\] 
be a point on the basis of $T(t_{i+1},x_{i+1})$ which satisfies~\eqref{E:uonbasis} where $(\bar t,\bar x)=(t_{i+1},x_{i+1})$.
Thus
\begin{equation}
\label{E:uonthenodes}
|u(t_{i},x_{i}) - u( t_{i-1}, x_{i-1})| \leq G (t_{i}-t_{i-1})
\qquad
i=1,\dots,k.
\end{equation}

By the choice of $x_{i-1}$, for every $k$ the slope
\[
\lambda_{i,k}=k(x_{i}-x_{i-1})
\] 
of each segment joining $(t_{i-1},x_{i-1})$, $(t_{i},x_{i})$ satisfies
\begin{equation}
\label{E:slopes}
\lambda(t_{i},x_{i})-M\omega\left(k^{-1}\right) 
\leq 
\lambda_{i,k}
\leq 
\lambda(t_{i},x_{i})+M\omega\left(k^{-1} \right)  ,
\end{equation}
It is in particular uniformly bounded by $L+M\omega(1)$.
By Ascoli-Arzel\`a theorem the piecewise affine curve $\gamma_{k}$ with edges $\{(t_{i},x_{i})\}_{i=0}^{k}$ converge uniformly as $k\uparrow\infty$, up to subsequence, to a continuous curve $\gamma$.
As $\lambda$ is continuous, Equation~\eqref{E:slopes} implies that the limit curve $\gamma_{}$ is also Lipschitz continuous with slope
\[
\dot\gamma_{}(t)=\lambda(t,\gamma(t)).
\]
This just means that we approximated a backward characteristic curve.

\step{Lipschitz continuity of $u$ along the curve}
For each $k\in\N$ set
\[
u_{k}(t_{i}):=u(t_{i},x_{i})
\qquad
i=0,\dots ,k
\]
and define a function $u_{k}(t)$ linear in each interval $(t_{i-1}, t_{i})$, $i=1,\dots,k$.
Equation~\eqref{E:uonthenodes} implies that $u_{k}(t)$ is $G$-Lipschitz continuous.
By the continuity of $u$ and by the uniform convergence to $\gamma$ of the piecewise affine paths $\gamma_{k}$ with edges $\{(t_{i},x_{i})\}_{i=0}^{k}$, the function $u_{k}(t)$ converges uniformly to $u(i_{\gamma}(t))$: one has therefore that $u(i_{\gamma}(t))$ is itself $G$-Lipschitz continuous.

\step{Forward characteristic}
We give two explanations for this step.
First, Lemma~\ref{L:nodissipation2} ensures that there is no entropy dissipation: one can thus reverse the time.
Applying the above procedure for the reversed time one finds a forward characteristic curve.
If one does not want to apply that strong lemma, it is enough being able to construct through each point $(t,x)\in(0,1)\times(-1,1)$ a backward characteristic $\gamma_{(t,x)}(s)$ along which $u$ is $G$-Lipschitz continuous. Having that, the function
\[
\gamma(t) :=\inf\big\{ x\ :\ \gamma_{(t,x)}(0)\geq 0\big\}, \qquad t\in(0,1)
\]
can be verified to be a characteristic curve passing through the origin.
Moreover, it is the uniform limit of characteristics along which $u$ is $G$-Lipschitz continuous; in particular, by the continuity, $u$ is therefore $G$-Lipschitz continuous along $\gamma$  itself.
\end{proof}

\begin{corollary}
\label{C:distrtoLagrangian}
Suppose $u$ is a continuous distributional solution of the PDE
\[
\pt u(t,x) + \px (f(u(t,x))) =  \gd(t,x),
\quad \norm{\gd}_{\distrBdd}\leq G.
\]
Then $u$ is also a Lagrangian solution, with a Lagrangian source $\gx$ bounded by $G$.
\end{corollary}

\begin{proof}
Lemma~\ref{L:nodissipation2} yields that $u$ is entropic. One can then apply Lemma~\ref{L:Excharlip}, providing through any point a characteristic along which $u$ is $G$-Lipschitz continuous. Lemmas~\ref{L:lipcharisLagr},~\ref{L:hogdipdatx} finally show that $u$ is a Lagrangian solution.
\end{proof}

\appendix

\section{Three sufficient conditions for the Lagrangian formulation}
\label{S:threesuffcond}

Given a continuous function $u$, we consider here some sufficient conditions for satisfying the Lagrangian formulation.
The section extends constructions in~\cite[\S~Appendix]{BCSC}.

\subsection{A dense set of characteristics}
\label{Ss:lagrparam}
Fix a continuous function $u$.
For having a Lagrangian parameterization along which $u$ is $G$-Lipschitz continuous, one clearly needs through each point of the domain a characteristic curve along which $u$ is $G$-Lipschitz continuous.
We prove here that this is sufficient.
This is Lemma~\ref{L:lipcharisLagr} in the introduction.

\begin{proof}[Proof of Lemma~\ref{L:lipcharisLagr}]
Simplify the domain to $u\in C_{\rc}(\R^{+}\times\R)$, as it is a local argument, and let $G>0$.
We assume that there exists a curve $\gamma_{(t,x)}(s)$ through each point $(t,x)$ of a dense subset of $\R^{+}\times\R$ such that $s\mapsto u(i_{\gamma_{(t,x)}}(s))$ is $G$-Lipschitz continuous.
We are going to modify these characteristics in order to provide a Lagrangian parameterization.

Consider an enumeration $\{(x_{r_{k}},y_{r_{k}})\}_{k\in\N}$ of a countable set of points, dense in the upper plane $\R^{+}\times\R$, where the characteristic curves are given by hypothesis.
We associate recursively to each point of this set a characteristic curve $\gamma_{k}(s)$ and we define a linear order among those:
\begin{itemize}
\item Let $\gamma_{1}(s)=\gamma_{(t_{r_{1}},x_{r_{1}})}(s)$. We define, for $k\in\N$,
\[
r_{k} \preceq r_{1}
\text{ if $x_{r_{k}}\leq\gamma_{1}(t_{r_{k}})$},\quad
r_{1}\preceq {r_{k}} \text{ if $x_{r_{k}}\geq\gamma_{1}(t_{r_{k}})$.}
\]
\item Let $h\in\N$. We define the new characteristic curve $\gamma_{h+1}(s)$ through $(t_{r_{h+1}},x_{r_{h+1}})$ in order to preserve the order relation that we are establishing:
\[
\gamma_{h+1}(s)=\min_{r_{h}\preceq r_{k}, k\leq h}\left\{ \gamma_{k}(s), \max_{r_{\ell}\preceq r_{h}, \ell\leq h}\left\{ \gamma_{\ell}(s), \gamma_{(t_{r_{h+1}},x_{r_{h+1}})}(s)\right\}\right\}.
\]
As $\gamma_{k}(s)$, for $k\leq h$, and $\gamma_{(t_{r_{h+1}},x_{r_{h+1}})}(s)$ are characteristic curves along which $u$ is $G$-Lipschitz continuous by hypothesis, then also $\gamma_{h+1}(s)$ is a characteristic curve along which $u$ is $G$-Lipschitz continuous.
We set then, for $k\in\N$,
\[
r_{k} \preceq r_{h+1}
\text{ if $x_{r_{k}}\leq\gamma_{h+1}(t_{r_{k}})$},\quad
r_{h+1}\preceq {r_{k}} \text{ if $x_{r_{k}}\geq\gamma_{h+1}(t_{r_{k}})$.}
\]
By construction it extends the relation defined at the previous steps.
\end{itemize}

The set of uniformly Lipschitz continuous curves
\[
\mathcal C =\left\{\gamma_{k}(s) \right\}_{k\in\N}
\]
is totally ordered and the images of these curves are dense in $\R^{+}\times\R$. We can complete this set in the uniform topology: the curves that we introduce with the closure are still characteristic curves because of the continuity of $f'(u)$; as well, $u$ is $G$-Lipschitz continuous along them and they preserve the order, in the sense that any two curves do not cross each other but always lie on a fixed side, when they differ.
If $\{q_{k}\}_{k\in\N}$ is an enumeration of the rational numbers, the map
\begin{align*}
\theta: \clos({\mathcal C}) &\to \R
\\
\gamma\quad  &\mapsto \sum_{k=0}^{\infty}\frac{\gamma(q_{k})}{2^{-k}}
\end{align*}
is continuous and strictly order preserving.
In particular, it is invertible with continuous inverse.

One can then verify that a Lagrangian parameterization is provided by
\[
\chi(s,y)=[\theta^{-1}(y)](s) \qquad \text{for $s\in \theta\left( \clos({\mathcal C}) \right)$}.
\]
By construction $t\mapsto U(t,y)=u(t,\chi(t,y))$ is $G$-Lipschitz continuous for each $y$ fixed: the thesis thus follows by Lemma~\ref{L:hogdipdatx}.
\end{proof}

\subsection{Lipschitz continuity along characteristics}
\label{Ss:lipalongchar}
Fix a continuous function $u$.
For having that $u$ is a Lagrangian solution, one clearly needs that  $t\mapsto u(t,\chi(t,y) )$ is Lipschitz continuous, uniformly in the $y$ parameter, for some Lagrangian parameterization $\chi$. We prove here that this is sufficient.
This is Lemma~\ref{L:hogdipdatx} in the introduction.

We are not concerned here with the compatibility of the source terms.

\begin{proof}[Proof of Lemma~\ref{L:hogdipdatx}]

Simplify the domain to $u\in C_{\rc}(\R^{+}\times\R)$, as it is a local argument, and let $G>0$.
We want to show that if there exists a Lagrangian parameterization $\chi$ such that
\[
\text{for all $y\in \R$}\qquad
-G\leq\exists \ddt u(t,\chi(t,y) ) \leq G
\qquad\text{in $\D(i_{\chi(y)}^{-1}(\Omega))$}
\]
then one can find a function $\gx\in\pwx(\Omega)$ such that
\[
\text{for all $y\in \R$}\qquad
\ddt u(t,\chi(t,y) ) = \gx(t, \chi(t,y))
\qquad\text{in $\D(i_{\chi(y)}^{-1}(\Omega))$}.
\]
Set $U(t,y)=u(t,\chi(t,y) )$ and consider $\Gfr\in \pw(\R^{+}\times\R)$ such that, in the $(t,y)$-half plane,
\[
\Gfr(t,y)= \pt U(t,y)
\qquad\text{in $\D(i_{\chi(y)}^{-1}(\Omega))$ for each $y\in\R$}.
\]
We want to show that it can be chosen of the form $\Gfr(t,y)=\gx(t, \chi(t,y))$ for some $\gx$, which means that it is essentially single valued on the level sets of $\chi$.
Fixed $y$, we show the following: the set of times $t$ where $u$ has a Lebesgue point of classical differentiability i) both along the characteristic curve $\gamma(t)=\chi(t,y)$ ii) and also along another characteristic curve $\bar \gamma(t)$, lying on a fixed side of $\gamma(t)$, iii) with two different values of the derivative, are at most countable.
This is enough since characteristics of a same Lagrangian parameterization are by definition ordered.

Let $\varepsilon,\sigma>0$.
By a reduction argument it suffices to show the following {\bf claim}: the set $S(y)=$
\begin{equation}
\label{E:lemmaaccessorio}
\begin{split}
\left\{
t:\qquad \left|\frac{U(t+h,y)-U(t,y)}{h}-\Gfr(t,y) \right| < \varepsilon, \quad \exists\gamma(s)\text{~characteristic with $\gamma(t)=\chi(t,y)$,}\right.\\ %
\left.\gamma(t+h)\leq\chi(t+h,y)  \text{~and~}\frac{u(t+h,\gamma(t+h))-u(t,\gamma(t))}{h}-\Gfr(t,y) > \varepsilon,  \qquad  \forall |h|\leq\sigma %
\right\}.
\end{split}
\end{equation}
does not contain two points $t_{1}, t_{2}$ closer than $\sigma$.
Indeed, if we are comparing the value of the derivative of $u$ along different characteristics of the parameterization $\chi$, then an order condition is satisfied among characteristics. Moreover, as we consider Lebesgue points of differentiability, with different values for the derivative of $u$ along $\chi_{y}^{\reu}(s)$ and $\gamma(s)$, up to a countable covering we are dealing with sets like~\eqref{E:lemmaaccessorio}.

We prove the claim by contradiction: let
\[
t_{1}, t_{2}\in S(y), \qquad  0<t_{2}-t_{1}<\sigma.
\]
The definition~\eqref{E:lemmaaccessorio} of $S(y)$ provides curves $\gamma_{1}$, $\gamma_{2}$ which intersect at times respectively $t_{1}$, $t_{2}$ 
\[
\gamma_{0} (s):= \chi(s,y) 
\]
and which for $t_{1}\leq s\leq t_{2}$ satisfy the additional properties
\begin{align*}
&\gamma_{1}(s)\leq \gamma_{0}(s),  &&  {u(s,\gamma_{1}(s))-u(t_{1},\gamma_{1}(t_{1}))}> \left( \Gfr(t,y)  + \varepsilon \right) (s-t_{1}),
\\
&\gamma_{2}(s)\leq \gamma_{0}(s),  &&u(t_{2},\gamma_{2}(t_{2})) -u(s,\gamma_{2}(s))> \left( \Gfr(t,y)    +\varepsilon \right) (t_{2}-s).
\end{align*}
By the ordering imposed in~\eqref{E:lemmaaccessorio} and by the uniform Lipschitz continuity implied by the fact that they are characteristics, the curves $\gamma_{1}(s), \gamma_{2}(s)$ necessarily meet at some time $\bar t\in[t_{1},t_{2}]$.
One can then compute the difference $U(t_{2},y)-U(t_{1},y)$ in two ways:
\begin{itemize}
\item applying the incremental relation in~\eqref{E:lemmaaccessorio} relative to $\chi$, which gives
\begin{align*}
U(t_{2},y)- U(t_{1},y) &= U(t_{2},y)   -U(\bar t,y) + U(\bar t,y)- U(t_{1},y)
\\ &< \Gfr(t_{2},y) (t_{2}-\bar t)+\Gfr(t_{1},y) (\bar t-t_{1})+2\varepsilon;
\end{align*}
\item applying the incremental relation in~\eqref{E:lemmaaccessorio} relative to $\gamma_{1}, \gamma_{2}$: denoting by $\bar x$ the value $\gamma_{1}(\bar t)=\gamma_{2}(\bar t)$ when $\gamma_{1}$ and $ \gamma_{2}$ intersect one has
\begin{align*}
U(t_{2},y)- U(t_{1},y) &= u(t_{2},\gamma_{2}(t_{2}))   -u(\bar t,\bar x) + u(\bar t,\bar x)- u(t_{1},\gamma_{1}(t_{1}))
\\ & > \Gfr(t_{2},y) (t_{2}-\bar t)+\Gfr(t_{1},y) (\bar t-t_{1})+2\varepsilon.
\end{align*}
\end{itemize}
The estimates that we obtain in the two ways are not compatible: we reach a contradiction. 
\end{proof}

\subsection{Stability of the Lagrangian formulation for uniform convergence of \texorpdfstring{$u$}{u}}
\label{Ss:stabilityparam}
We state for completeness that Lagrangian solutions are closed w.r.t.~uniform convergence, provided the sources are uniformly bounded.
We include this for completeness but it follows easily by the previous analysis of the section. 
\begin{corollary}
\label{C:unifConverengeLagr}
Let $G>0$ and $u_{k}(t,x)$ be a sequence of continuous Lagrangian solutions of
\begin{equation*}
\tag{\ref{E:basicPDE}}
\pt u_{k}(t,x) + \px (f(u_{k}(t,x))) =  \gd_{k}(t,x)
\qquad f\in C^{2}(\R),
\qquad 
|\gd_{k}(t,x)|\leq G.
\end{equation*}
If $u_{k}$ converges uniformly to $u$, then $u$ is a Lagrangian solution with source term bounded by $G$.
\end{corollary}

\begin{proof}
We verify that through every point $(t,x)\in\R^{+}\times\R$ there exists a characteristic curve $\gamma(s)$ such that $u(i_{\gamma}(s))$ is $G$-Lipschitz continuous: Lemma~\ref{L:lipcharisLagr} then provides a Lagrangian parameterization along which $u$ is $G$-Lipschitz continuous, and Lemma~\ref{L:hogdipdatx} gives the thesis.

As $\{u_{k}\}_{k\in\N}$ are Lagrangian solutions of~\eqref{E:basicPDE} with sources uniformly bounded by $G$, one can find for each $k\in\N$ a characteristic curve $\gamma_{k}(s)$ of $u_{k}$ through $(t,x)$ satisfying
\begin{equation}
\label{E:lemmadiuncorollario}
|u_{k}(i_{\gamma_{k}}(r))-u_{k}(i_{\gamma_{k}}(s))| \leq G|r-s|.
\end{equation}
The family $\{\gamma_{k}(s)\}_{k\in\N}$ is locally equi-Lipschitz continuous and equi-bounded, as $\gamma_{k}(t)=x$.
By Ascoli-Arzel\`a theorem this family has a subfamily uniformly convergent to a function $\gamma(s)$.
From the uniform convergence of the continuous functions $u_{k}$ and $\gamma_{k}$, the relation
\[
\gamma_{k}(r)-\gamma_{k}(s)=\int_{r}^{s}u_{k}(i_{\gamma_{k}}(q))dq
\] 
goes tot he limit and it implies that $\gamma$ is characteristic curve for $u$. Moreover, also~\eqref{E:lemmadiuncorollario} goes to the limit and it yields that $u$ is $G$-Lipschitz continuous along $\gamma$.
\end{proof}

\begin{thenomenclature} 

 \nomgroup{A}

  \item [{$[\cdot]_{\lambda}$, $[\cdot]_{\chi}$, $[\cdot]$}]\begingroup Projections on, $\pwu(X)$, $\pwx(X)$, $\distrBdd(X)$ respectively, Notation~\ref{notation:projection}\nomeqref {1.1}
		\nompageref{5}
  \item [{$\chi$}]\begingroup Lagrangian parameterization for a continuous solution $u$ to \eqref{E:basicPDE}, Defintion~\ref{D:LagrangianParameterization}\nomeqref {1.1}
		\nompageref{4}
  \item [{$\clos(\cdot)$}]\begingroup Closure of a set\nomeqref {1.1}
		\nompageref{7}
  \item [{$\D(\Omega)$}]\begingroup Distributions on $\Omega$, Notation~\ref{N:variousnotations}\nomeqref {1.1}
		\nompageref{5}
  \item [{$\distrBdd(X)$}]\begingroup Bounded functions on $X$ identified $\Ll^{2}$-a.e., Notation~\ref{N:variousnotations}\nomeqref {1.1}
		\nompageref{5}
  \item [{$\Dt, \Dx$}]\begingroup Partial derivatives of a function of bounded variation, Notation~\ref{N:derivatives}\nomeqref {1.1}
		\nompageref{3}
  \item [{$\gamma, i_{\gamma}$}]\begingroup Characteristic curve, Definition~\ref{D:charactcurves}\nomeqref {1.1}
		\nompageref{4}
  \item [{$\gfr$, $\gu$, $\gx$}]\begingroup Functions beloning to $\pw(X)$, $\pwu(X)$, $\pwx(X)$ respectively, Notation~\ref{N:variousnotations}\nomeqref {1.1}
		\nompageref{5}
  \item [{$\gsimple,\gd$}]\begingroup Distributional, bounded source term for the balance law~\eqref{E:basicPDE}\nomeqref {1.1}
		\nompageref{5}
  \item [{$\infl(f)$}]\begingroup Inf{}lection points of $f$, Definition~\ref{D:inf{}lf}\nomeqref {1.1}
		\nompageref{7}
  \item [{$\lambda$}]\begingroup The composite function $f'\circ u$, Notation~\ref{N:basic}\nomeqref {1.1}
		\nompageref{3}
  \item [{$\Ll^{1}$, $\Ll^{2}$}]\begingroup 1- or 2-dimensional Lebesgue measure\nomeqref {1.1}
		\nompageref{3}
  \item [{$\mathcal M(X)$}]\begingroup Radon measures on $X$, Notation~\ref{N:variousnotations}\nomeqref {1.1}
		\nompageref{5}
  \item [{$\Omega$}]\begingroup Open subset of $\R^{+}\times\R$, if needed connected\nomeqref {1.1}
		\nompageref{5}
  \item [{$\pat{},\pax{}$}]\begingroup Classical partial derivatives, Notation~\ref{N:derivatives}\nomeqref {1.1}
		\nompageref{3}
  \item [{$\pt,\px$}]\begingroup Distributional partial derivatives, Notation~\ref{N:derivatives}\nomeqref {1.1}
		\nompageref{3}
  \item [{$\pw(X)$}]\begingroup Functions defined pointwise on $X$, Notation~\ref{N:variousnotations}\nomeqref {1.1}
		\nompageref{5}
  \item [{$\pwu(X)$}]\begingroup Functions coinciding $\Ll^{1}$-a.e.~on characteristics of $u$, Notation~\ref{N:variousnotations}\nomeqref {1.1}
		\nompageref{5}
  \item [{$\pwx(X)$}]\begingroup Functions coinciding $\Ll^{1}$-a.e.~on the Lagrangian parameterization $\chi$, Notation~\ref{N:variousnotations}\nomeqref {1.1}
		\nompageref{5}
  \item [{$\varphi_{t}^{\red}(x)$}]\begingroup Restriction of a function $\varphi(t,x)$ to the second coordinate, Notation~\ref{N:restrictions}\nomeqref {1.1}
		\nompageref{3}
  \item [{$\varphi_{x}^{\reu}(t)$}]\begingroup Restriction of a function $\varphi(t,x)$ to the first coordinate, Notation~\ref{N:restrictions}\nomeqref {1.1}
		\nompageref{3}
  \item [{$C(\Omega)$}]\begingroup Continuous functions on $\Omega$, see also $C_{b}, C^{k}_{}, C^{k}_{\rc}, C^{k, 1/\alpha}$ in Notation~\ref{N:variousnotations}\nomeqref {1.1}
		\nompageref{5}
  \item [{$f$}]\begingroup Flux function for the balance law~\eqref{E:basicPDE}\nomeqref {1.1}
		\nompageref{3}
  \item [{$u$}]\begingroup Continuous solution, Noation~\ref{N:basic}\nomeqref {1.1}
		\nompageref{3}
  \item [{$X$}]\begingroup Subset of $\R^{+}\times\R$, usually Borel.\nomeqref {1.1}
		\nompageref{5}

\end{thenomenclature}

\section*{Acknowledgement}
The authors wish to thank Francesco Bigolin and Francesco Serra Cassano, from the University of Trento, for interesting discussions which further motivated this later paper.
{We are grateful to the anonymous referee for interesting questions improving our bibliographical references.}
This work was supported by the Centro di Ricerca Matematica `Ennio De Giorgi' (Pisa), the EPSRC Science and Innovation award to the OxPDE (EP/E035027/1), GNAMPA of the Istituto Nazionale di Alta Matematica (INdAM), the PRIN national project ``Nonlinear Hyperbolic Partial Differential E\-qua\-tions, Dispersive and Transport Equations: theoretical and applicative aspects''.

\end{document}